\documentclass[a4paper,12pt]{amsart}

\usepackage{amsmath,amsfonts,amsthm,url}
\usepackage{graphicx}
\usepackage[notref, notcite,final]{showkeys}

\usepackage{slashbox}
\usepackage[table]{xcolor}
\newtheorem{theorem}{Theorem}[section]
\newtheorem{conjecture}[theorem]{Conjecture}
\newtheorem{corollary}[theorem]{Corollary}
\newtheorem{lemma}[theorem]{Lemma}
\newtheorem{proposition}[theorem]{Proposition}

\theoremstyle{definition}

\theoremstyle{remark}
\newtheorem{remark}[theorem]{Remark}
\newtheorem{example}[theorem]{Example}

\newcommand{\NN}{\mathbb{N}}
\newcommand{\ZZ}{\mathbb{Z}}
\newcommand{\QQ}{\mathbb{Q}}
\newcommand{\RR}{\mathbb{R}}

\newcommand{\K}[3]{g_{#1, #2, #3}}                				
\newcommand{\RK}[3]{\overline{g}_{#1, #2, #3}}    			
\newcommand{\rk}[1]{{r}_{#1}}
\newcommand{\RKO}[1]{\overline{g}_{#1}}           			
\newcommand{\SSK}[3]{\overline{\overline{g}}_{#1, #2, #3}}      

\newcommand{\ep}{\emptyset} 							
\newcommand{\cut}[1]{\overline{#1}}
\newcommand{\cutt}[1]{\widehat{#1}}
\newcommand{\cutlmn}{\cut{\lambda},\cut{\mu},\cut{\nu}}
\newcommand{\cutcutlmn}{\cut{\cut{\lambda}},\cut{\cut{\mu}},\cut{\cut{\nu}}}

\newcommand{\cbox}{\cellcolor{gray!40} 145}
\newcommand{\vertexop}[2]{\Gamma_{(#1|#2)}} 	
\newcommand{\SetL}{\mathcal{L}}                         
\newcommand{\vvar}{{\bf t}}
\newcommand{\Scol}{\Phi^-}  					
\newcommand{\scol}{\phi^-}  					
\newcommand{\polcol}{Q^-}     					
\newcommand{\Srow}{\Phi^+}    				
\newcommand{\srow}{\phi^+}    					
\newcommand{\polrow}{Q^+}     				

\newcommand{\op}[2]{\oplus (#1|#2)}  			
\newcommand{\Sym}{\textit{Sym}_{\QQ}} 			
\newcommand{\scalar}[2]{\left\langle  #1 \, \middle| \, #2  \right\rangle}

\author{Emmanuel Briand}
\thanks{E. Briand and M.Rosas are partially supported by projects MTM2013--40455--P, FQM--333, P12--FQM--2696 and FEDER} 
\author{Amarpreet Rattan}
\author{Mercedes Rosas}
\title[Kronecker coefficients]{On the growth of the Kronecker coefficients.}
\date{\today}

\begin{document}

\maketitle

\begin{abstract} We study the rate of growth experienced by the Kronecker coefficients as we add cells to the rows and columns
indexing partitions. We do this by moving to the setting of the reduced Kronecker coefficients.
\end{abstract}

\section{Introduction}

The Kronecker coefficients $\K{\lambda}{\mu}{\nu}$ are fundamental constants in Representation Theory. 
They describe how irreducible representations of $GL(V \otimes W)$ split, when viewed as representations of $GL(V) \times GL(W)$. They are also the structural constants for the tensor products of irreducible representations of the symmetric groups.  

In spite of their importance, very little is known about the Kronecker
coefficients, and this leaves some fundamental questions unanswered.  For
example, are the Kronecker coefficients described by a positive combinatorial
rule, akin to the Littlewood--Richardson rule
\cite{Ballantine:Orellana,Blasiak:hook,Liu}?  How difficult is it, algorithmically, to compute Kronecker coefficients  \cite{GCT6,  Burgisser:Ikenmeyer, BOR-CC,PakPanova:complexity}, or to determine they are nonzero \cite{Ressayre, Ikenmeyer:Mulmuley:Walter:Vanishing}?  Remarkably, this latter problem relates  the Kronecker coefficients with the \emph{quantum marginal problem} in Quantum Information Theory \cite{Klyachko, Christandl:Mitchison, Christandl:Harrow:Mitchison}.

A feature of the Kronecker coefficients that has been studied recently is the
\emph{stability} phenomenon: the fact that some sequences of Kronecker
coefficients are eventually constant. The first example of such a behavior was
observed by Murnaghan in 1938 \cite{Murnaghan:1938}. The Kronecker coefficients
$\K{\lambda}{\mu}{\nu}$ are indexed by triples of partitions
$(\lambda,\mu,\nu)$, and Murnaghan's stable sequences are obtained by
incrementing the first part of all three partitions at each step. Their limit
values (the \emph{reduced}, or \emph{stable} Kronecker coefficients) are
interesting objects in their own right.  As seen in \cite{BOR-JA}, they contain
enough information to recover the value of the Kronecker coefficients and are believed to be simpler to understand.
For example, it is conjectured that they satisfy the \emph{saturation property} \cite{Kirillov:saturation, Klyachko:saturation}, and they have been used to find efficient formulas for computing some Kronecker coefficients \cite{BOR-CC}.

 Many more sequences of Kronecker coefficients are stable: large families have
 been produced by means of methods from geometry \cite{Manivel,
 Manivel:asymptotics1, Manivel:asymptotics2}, enumerative combinatorics
 \cite{Vallejo:tomography, Vallejo} or symmetric functions calculations
 \cite{PakPanova:bounds}. These stable sequences of Kronecker coefficients have
 general term of the form $\K{\lambda + n \alpha}{\mu + n\beta}{\nu+n\gamma}$, where we
 add, and multiply a partition by a scalar, as it is usually done for vectors.

 Murnaghan's case corresponds to $\alpha=\beta=\gamma=(1)$. 
 
 In \cite{Stembridge}, it was conjectured that $\K{\lambda + n \alpha}{\mu +
 n\beta}{\nu+n\gamma}$ stabilizes (for any $\lambda$, $\mu$, $\nu$) if and only
 if $\K{\alpha}{\beta}{\gamma} = 1$. This was proved in
 \cite{SamSnowden:Stembridge}. These stability phenomena are, as an aside, the
 prototype for the very general \emph{representation stability} phenomenon
 unveiled in algebraic topology;  see \cite{Church:Farb:stability,ChurchEllenbergFarb:FI, SamSnowdenGeneralizedStability}.

In this paper we present two new results related to the stability of Kronecker
coefficients. The first one is indeed a result of stability, but the sequence
that we consider is not of the type $\K{\lambda + n \alpha}{\mu +
n\beta}{\nu+n\gamma}$. At each step, we simultaneously
increase the first row and the first column of the Young diagrams of all three indexing partitions. We call this phenomenon \emph{hook stability}. 
 
Note that this hook stability does not seem to fit straighforwardly in the
representation theory of fixed general linear groups, since it involves
sequences of Kronecker coefficients indexed by partitions with unbounded
lengths.

The second result is about the \emph{asymptotics} of some sequences of Kronecker coefficients of type $\K{\lambda + n \alpha}{\mu + n\beta}{\nu+n\gamma}$, that do not stabilize, but are shown to grow linearly.

We describe the relevant coefficients (the limits for hook stability, and the coefficients appearing in quasipolynomial formulas for the asymptotic estimates, for the result on linear growth) by means of generating series. 

Our  tools are the following:
\begin{enumerate}
\item \emph{Vertex operators on symmetric functions}.  Vertex operators on
	symmetric functions provide generating functions for Schur functions.
	They have been used widely by Thibon and his collaborators to establish
	several properties of stability.  See Section \ref{sec:vertop} for
	references and a basic treatment of vertex operators.
\item \emph{The $\lambda$--ring formalism for symmetric functions}.  This formalism is in
	fact a calculus on morphisms from the algebra of symmetric functions.
	See  Section \ref{sec:lambdaring} for basic definitions and references.
\item \emph{Schur generating series} will be used to encode families of constants indexed
	by several partitions by means of symmetric series in several sets of
	variables. Important structural constants for symmetric functions have
	very compact Schur generaing series when expressed within the
	Lambda--ring formalism: $\sigma[XY+XZ]$ for Littlewood--Richardson
	coefficients, and $\sigma[XYZ]$ for Kronecker coefficients. The
	coefficients introduced in this paper also have have simple Schur generating series. 
\end{enumerate}
The two sets of results in this paper (hook stability  and linear growth) are obtained by first considering stability properties and linear growth for families of 
reduced Kronecker coefficients, and then translating the results
obtained to Kronecker coefficients.

The two sets of results for reduced Kronecker coefficients are obtained the same way: by simplifying Schur generating series for sequences of reduced Kronecker coefficients  by means of vertex operators. Because we have at our disposal two conjugate vertex operators (one related to first row increasing, the other to first column increasing) we simultaneously get these two sets of results.

This article is structured as follows. In Section \ref{prelim} we introduce the
basic tools used in this article. In particular, we review the two vertex operators that allow us to increase the sizes of the first row  and column of a partition.

Section \ref{rkro} presents the reduced Kronecker coefficients: it
includes an elementary proof of Brion's formula \cite{Brion},  which we have not seen in the literature,   and an elementary derivation of the generating function for the reduced Kronecker coefficients indexed by one row (and one column) shapes.

Section \ref{a-factorization}  provides the main technical lemma that allows us
to factor a symmetric function (polynomial) out of some symmetric series. This
lemma is applied twice in Sections \ref{hook-stability} and \ref{section2row},
once with each of the two conjugate vertex operators.

In Section  \ref{hook-stability}, we prove stability for the sequences of
reduced Kronecker coefficients whose indexing partitions have their first column
growing (Section \ref{hook reduced}). We deduce in Section \ref{hook subsection}
the hook stability property for Kronecker coefficients. Another approach to
proving this property is explored in Section \ref{alternative subsection}. We are not able to get an alternative proof of the hook stability property through this approach, but are led to the conjecture that Kronecker coefficients weakly increase when 
incrementing at the same time the first row and the first column of the diagram of each of their three indexing partitions (Conjecture \ref{conj 111}).

In Section \ref{section2row}, we study the effect of the growth of the first
rows of the partitions indexing a reduced Kronecker coefficient and obtain
linear quasipolynomial formulas when these first rows are big enough. This also provides asymptotic linear quasipolynomial formulas for some sequences of Kronecker coefficients $\K{\lambda +n \alpha}{\mu + n \beta}{\nu + n \gamma}$ where, together with other conditions, the partitions $\alpha$, $\beta$ and $\gamma$ have at most two parts.

Schur generating functions for the limit values $\SSK{\lambda}{\mu}{\nu}$ in the hook stability property, and the coefficients of the quasipolynomial formulas of Section \ref{section2row}, are derived in Section \ref{generating series}.

 Section \ref{final-remarks} reviews some examples appearing in the literature, that study the effect of increasing the sizes of the other rows and columns.

 Two appendices round up the results. In Appendix \ref{apptab}, we provide a
 table with some constants appearing in Theorem \ref{prop:ABC}. Appendix
 \ref{appbounds} contains a proof of  some stability bounds for the hook stability described by Theorem \ref{thm:hook stab}.


 \section{Preliminaries}\label{prelim}

\subsection{Partitions}\label{partitions}

A \emph{partition of $n$} is a weakly descending sequence of non-negative integers
whose sum is $n$. Two  partitions that differ by a string of zeros are
considered to be the same. The positive terms in a partition are called its
\emph{parts}, and the \emph{length $\ell(\lambda)$} of the partition $\lambda$
is  defined as the number of parts. The weight $|\lambda|$ of the partition
$\lambda$ is the sum of its parts. The conjugate of a partition $\lambda$ will
be denoted $\lambda'$, and with parts $\lambda_1', \lambda_2',
\ldots$.  The empty partition will be denoted by $\ep$.

Let $\cup$ and $+$ be the standard operations on partitions, as defined in \cite[I.\S1]{Macdonald}. If $n$ is a nonnegative integer and $\lambda=(\lambda_1, \lambda_2,\ldots,\lambda_k)$ is a partition, then $n \lambda$ is the dilation of $\lambda$ by a factor $n$, that is the partition $(n \lambda_1, n\lambda_2, \ldots, n \lambda_k)$.

Let  $\cut{\lambda}$ be the partition obtained after removing the first term of $\lambda$. This operation can be iterated:  $\cut{\cut{\lambda}}$ is the partition obtained from $\lambda$ by removing the first two terms.
 
Let  $\cutt{\lambda}$ be the partition obtained after removing the first row and the first column in the diagram of $\lambda$. 
 
The sequence defined by prepending a first term $a$ to the partition $\lambda$ will be denoted $(a,\lambda)$. The resulting sequence $(a,\lambda_1,\lambda_2, \ldots)$ is not necessarily a partition since we may have that $a < \lambda_1$.
Given  an integer $N$, we denote by $\lambda[N]$ the sequence
$(N-|\lambda|,\lambda)$, which is also not necessarily a partition.

Finally, for any non--empty partition $\lambda$, we will write $\lambda\op{a}{b}$ for $\lambda +(a) \cup (1^b)$.

For example, if $\lambda=(8,3,3,1)$, then we have that $\cut{\lambda}=(3,3,1)$, $\cut{\cut{\lambda}}=(3,1)$, $\cutt{\lambda}=(2,2)$, $\lambda[20]=(5,\lambda)=(5,8,3,3,1)$ (not a partition),  $\lambda[25]=(10,\lambda)=(10,8,3,3,1)$  and  $\lambda\op{7}{4}=(15,3,3,1,1,1,1,1)$.

\subsection{Symmetric functions, Schur functions and Jacobi--Trudi determinants}
For $\lambda$ a finite sequence of integers of length $n$, we define 
\[
s_{\lambda}=\det \left( h_{\lambda_j+i-j} \right)_{i,j=1\ldots n},
\]
where $h_0=1$ and $h_k=0$ for $k<0$. 

The Jacobi-Trudi formula implies that when $\lambda$ is a partition then
$s_{\lambda}$ is the Schur function indexed by $\lambda$. Since rearranging the
columns of the above determinant suffices to order the parts of any sequence of integers, we always obtain for $s_{\lambda}$ either  a Schur function (up to  sign), or zero. For example, 
$
s_{\lambda_1,\lambda_2}=\left| 
\begin{matrix}
h_{\lambda_1}  & h_{\lambda_2-1}\\
h_{\lambda_1+1}& h_{\lambda_2}
\end{matrix}
\right|,
$ 
$s_{1,2}=0
$, and 
$
s_{1,3}=-s_{(2,2)}
$.


Let $\Sym=\Sym(X)$ be the algebra of symmetric functions with rational coefficients, with underlying alphabets $X=\{x_1,x_2,\ldots\}$. We denote by  $\scalar{ }{}$ or $\scalar{ }{}_X$ the scalar product  on $\Sym$ defined by saying that the Schur functions are an orthonormal basis. For any symmetric function $f$, $f^{\perp}$ will denote the adjoint of multiplication by $f$.  The scalar product is conveniently extended whenever it makes sense. For instance 
\[
\scalar{\sum_{i=0}^{\infty} f_i t^i}{\sum_{j=0}^{\infty} g_j}=\sum_{i=0}^{\infty}\left( \sum_{j=0}^{\infty}\scalar{f_i}{g_j} \right)t^i
\]
if, for each $i$, $f_i$ is a symmetric function and for each $j$, $g_j$ is a homogeneous symmetric function of degree $j$.

We also consider symmetric functions in different alphabets (set of variables)
$X$, $Y$, $Z$. The scalar product is canonically extended to the algebras
$\Sym(X) \otimes_{\QQ} \Sym(Y)$, $\Sym(X) \otimes_{\QQ} \Sym(Y) \otimes_{\QQ}
\Sym(Z)$ they generate, and denoted by $\scalar{}{}_{X,Y}$ and
$\scalar{}{}_{X,Y,Z}$.

\subsection{The $\lambda$--ring formalism for symmetric functions, and
specializations.}\label{sec:lambdaring}
Let $\mathcal{A}$ be any commutative algebra over a field $\mathcal{K}$ of characteristic zero

 Given a morphism of algebras $A$ from $\Sym$ to $\mathcal{A}$, the image of a symmetric function $f$ under $A$ will be denoted with $f[A]$ rather than $A(f)$ and called ``specialization of $f$ at  $A$''. 

Since the power sum symmetric functions $p_k$ ($k \geq 1$) generate $\Sym$ and are algebraically independent, the map 
\begin{equation}\label{bijection}
A \mapsto (p_1[A], p_2[A], \ldots)
\end{equation}
is a bijection from the set of all morphisms of algebras from $\Sym$ to $\mathcal{A}$ 
to the set of infinite sequences of elements from $\mathcal{A}$. This set of sequences is endowed with its operations of component-wise sum, product, and product by a scalar. The bijection \eqref{bijection} is used to lift these operations to the set of morphism from $\Sym$ to $\mathcal{A}$. This defines expressions like $f[A+B]$, $f[-A]$, $f[AB]$, $f[A/B]$ \ldots where $f$ is a symmetric function and $A$ and $B$ are two specializations, and more general expressions $f[P(A,B,\ldots)]$ where $P(A,B,\ldots)$ is a rational function in several specializations $A$, $B$ \ldots with coefficients in $\mathcal{K}$. Note that, by definition, for any power sum $p_k$ ($k \geq 1$), specializations $A$ and $B$ and scalar $z$,
\begin{align*}
&&p_k[A+B]=p_k[A]+p_k[B],&&p_k[AB]=p_k[A]p_k[B], &&p_k[zA]=z p_k[A].
\end{align*}

Here are some important specializations. The specialization at $-1$ is defined on power sums by $p_k[-1]=-1$ for all $k$. The specialization $\varepsilon$ is defined by $p_k[\varepsilon]=(-1)^k$ for all $k$.
The product of the two previous specializations is $-\varepsilon$ and fulfills
$p_k[-\varepsilon X]=(-1)^{k+1} p_k[X]$ for all $k$. As a consequence, the
transformation $f[X]  \mapsto f[-\varepsilon X]$ coincides with the standard involution $\omega$ defined by $\omega s_{\lambda}=s_{\lambda'}$ for all $\lambda$.
 There is also the specialization $X^{\perp}$ such that for any symmetric function $f$,
 $f[X^{\perp}]=f^{\perp}$, the adjoint of the multiplication by $f$ with respect
 to $\scalar{}{}_{X}$. 
 
\begin{lemma}\label{L1}
Let $\sigma[X]=\sum_{n \ge 0} h_n[X]$ be the generating function for the complete homogeneous symmetric functions in $X$. It has the following well-known properties: 
\begin{enumerate}

\item Given an alphabet $X$, 
\[
\sigma[X]=\prod_{x \in X} \frac{1}{1-x} \text{ and }\sigma[-X]=\prod_{x \in X}  (1-x).
\]
In particular for a single variable $t$, $\sigma[t]=1/(1-t)$ and $\sigma[-t]=1-t$;

\item Cauchy's Identity :  $\sigma[XY]=\sum_{\lambda} s_{\lambda}[X] s_{\lambda}[Y]$.

\item Given any two alphabets $A$ and $B$, $\sigma[A+B]=\sigma[A] \sigma[B].$

\item The adjoint of multiplication by $\sigma[AX]$  with respect to
	$\scalar{}{}_{X}$. It has the following effect:
$
\sigma[AX^{\perp}] f[X] = f[X+A].
$
\item As a particular case, we have the \emph{reproducing kernel property} of $\sigma[AX]$: for any symmetric function $f$,
$
\scalar{\sigma[AX]}{f[X]}=f[A].
$
\end{enumerate}
Standard references for these results are \cite{Macdonald} and \cite{Lascoux}. See also \cite{BMOR}.
\end{lemma}

Finally, it is well-known that using operations on alphabet, we can recover the
Littlewood--Richardson,  $c_{\lambda,\mu,\nu} $  and the Kronecker coefficients,
$g_{\lambda,\mu,\nu}$ :
\begin{align}
s_{\lambda}[X+Y]=\sum_{\mu,\nu} c_{\lambda, \mu,\nu}
s_{\mu}[X]s_{\nu}[Y]\notag\\
s_{\lambda}[XY]=\sum_{\mu,\nu} g_{\lambda,\mu,\nu}
s_{\mu}[X]s_{\nu}[Y].\label{def:kron}
\end{align}

While \eqref{def:kron} can be used to define the Kronecker coefficients, they can also be defined as follows.
Let $\lambda$ and $\mu$ be partitions of some integers.  Define a product $\ast$ on the ring of symmetric functions where the 
\begin{equation}\label{def:ast}
	p_\lambda \ast p_\mu = \delta_{\lambda, \mu} z^{-1}_\lambda p_\lambda,
\end{equation}
where, as usual, $z_\lambda = 1^{m_1} m_1! \cdots n^{m_n} m_n!$ and $m_i$ are
the number of parts of $\lambda$ equal to $i$.  Then, if $\nu$ is also a
partition, the Kronecker coefficients $g_{\lambda, \mu, \nu}$ are
\begin{equation*}
	s_\mu[X] \ast s_\nu[X] = \sum_\lambda g_{\lambda, \mu, \nu}
	s_\lambda[X].
\end{equation*}
It's clear from \eqref{def:ast}, that $g_{\lambda, \mu, \nu} = 0$ if $\lambda,
\mu$ and $\nu$ are not partitions of the same integer.

\subsection{Vertex operators for symmetric functions.}

\subsubsection{Vertex operators.}\label{sec:vertop}

The \emph{vertex operator for symmetric functions} $\vertexop{t}{X}$ is defined on the basis of Schur functions in $X$ by:
\[
\vertexop{t}{X}: s_{\alpha}[X] \mapsto \sum_{n \in \ZZ} s_{(n,\alpha)}[X] t^n
\]
where $t$ is an additional variable. 
Recall that Schur functions are defined using the Jacobi--Trudi identity in terms of the complete homogeneous basis, and that
$h_{n}$ is equal to zero when $n <0$.

From \cite[Lemma 3.1]{ThibonCarre}, this operator can be factorized as $\vertexop{t}{X}=\sigma[tX] \sigma\left[-\frac{1}{t} X^{\perp}\right]$ and therefore fulfills
\begin{equation}\label{factorizationGamma}
\vertexop{t}{X} f[X] = \sigma[t X] f\left[ X-\frac{1}{t}\right]
\end{equation}
for any symmetric function $f$. In particular, given any partition $\alpha$,
\begin{equation}\label{V1}
 \sum_{n \in \ZZ} s_{(n,\alpha)}[X] t^n =  \vertexop{t}{X} s_{\alpha}[X] = \sigma[t X] s_{\alpha}\left[ X-\frac{1}{t}\right].
\end{equation}
Alternatively, using the index $n$ for the weights of the partitions in the formal series instead of for the first parts, we have also
\begin{equation}\label{V2}
 \sum_{n \in \ZZ} s_{(n-|\alpha|,\alpha)}[X] t^n =  t^{|\alpha|} \vertexop{t}{X} s_{\alpha}[X] = \sigma[t X] s_{\alpha}\left[ t X-1 \right].
\end{equation}This follows from \eqref{factorizationGamma}, and the fact that Schur functions are homogeneous.

This vertex operator is a classical tool in the theory of symmetric functions used in particular by Jing (see for instance \cite{Jing:Spin}), and, for the study of various phenomena of stability, by Thibon and his collaborators \cite{Thibon,ThibonScharfWybourne, ThibonCarre, ThibonScharf:innerPlethysm,ThibonLascoux:stability}.  It is the generating series for Bernstein's creation operators introduced in \cite{Zelevinsky:LNM}. See also \cite[I.\S 5 Ex. 29]{Macdonald}. 

\subsubsection{Vertex operators for columns}

The vertex operator $\vertexop{t}{X}$ associates to any Schur function
$s_{\alpha}$ a generating series for the Schur functions $s_{(n,\alpha)}$
obtained by prepending a first part $n$ to $\alpha$. Let us build another
operator that associates to $s_{\alpha}$ a generating series for the Schur
functions $s_{\alpha+(1^n)}$ obtained by adjoining to the Young diagram of $\alpha$ a first column $(1^n)$. For this, apply the involution $\omega$ (that maps $s_{\alpha}$ to $s_{\alpha'}$), next the vertex operator $\vertexop{t}{X}$ (appends a first row to $\alpha'$) and then again $\omega$ (the new first row of $\alpha'$ becomes a new first column attached to $\alpha$). That is, our new operator is $\omega \vertexop{t}{X} \omega$. It sends any Schur function $s_{\alpha}$ to  $\omega \vertexop{t}{X} s_{\alpha'}$, which is equal to $\sum_n \omega s_{(n,\alpha')} t^n$. 

 Set $\lambda=(n,\alpha')$.  Recall that $s_{(n,\alpha')}$ is the Jacobi--Trudi determinant $\det(h_{\lambda_{j} +i-j})_{i,j}$ of order $\ell(\alpha')+1$. The involution $\omega$ exchanges  $h_k$ with  $e_k$. Therefore $\omega s_{(n,\alpha')}=\det(e_{\lambda_{j} +i-j})_{i,j}$. We will denote with $\tilde{s}_{(1^n|\alpha)}$ the value of  this determinant. When $n \geq \alpha'_1$, this determinant is equal to $s_{\alpha+(1^n)}$, the Schur function indexed by the partition obtained from $\alpha$ by adding a new column of size $n$ to its diagram. 
 For instance, $\vertexop{-\varepsilon t}{X}(s_{\emptyset})=\omega( \sigma[tX]),$ the generating function for the elementary symmetric functions.
 We have thus
\begin{align*}
\omega \vertexop{t}{X} \omega s_{\alpha}[X] 
&= \sum_{n \in \ZZ} \tilde{s}_{(1^n|\alpha)}[X] t^n \\
&= \sum_{n \geq \alpha'_1} s_{\alpha+(1^n)}[X] t^n + \text{ terms of degree $< \alpha'_1$ in $t$}.
\end{align*}
Since $\omega$ coincides with $f[X] \mapsto f[-\varepsilon X]$, 
\begin{align*}
(\omega \vertexop{t}{X} \omega)(f)
&=\omega \vertexop{t}{X} f[-\varepsilon X]
=\omega \sigma[tX]f\left[-\varepsilon \left(X-\frac{1}{t}\right)\right],\\
&=\sigma[-\varepsilon tX]f\left[-\varepsilon \left(-\varepsilon X-\frac{1}{t}\right)\right]=\sigma[-\varepsilon tX]f\left[X-\frac{1}{(-\varepsilon t)}\right],
\end{align*}
and thus
\[
\omega \vertexop{t}{X} \omega = \sigma[-\varepsilon t X] \sigma\left[-\frac{1}{-\varepsilon t} X^{\perp}\right].
\]
We will write $\vertexop{-\varepsilon t}{X}$ for $\omega \vertexop{t}{X} \omega$.

The operator $\vertexop{-\varepsilon t}{X}$ appears, for instance,  in \cite{Jing:Spin, JarvisYungVertex2, JingNie}. The operators $\vertexop{t}{X}$ and $\vertexop{-\varepsilon t}{X}$ are $V_\alpha(t)$ with $\alpha=1$ and $\alpha=-1$ respectively in the notations of \cite{JarvisYungVertex2}. They are $S(t)$ and $S^*(t)$ in the notations of \cite{JingNie}.

\section{Reduced Kronecker coefficients}\label{rkro}

\subsection{Murnaghan Stability}\label{Murnaghan}

Murnaghan observed \cite{Murnaghan:1938} that, for any triple of partitions
$\lambda$, $\mu$, $\nu$ of some positive integer $n$, the sequence of Kronecker coefficients $\K{\lambda+(m)}{\mu+(m)}{\nu+(m)}$ stabilizes (\emph{i.e.} is eventually constant).

Several classical proofs exist of this fact. It has been shown by Littlewood using invariant theory \cite{Littlewood:1958}, by Brion using geometric methods \cite[\S 3.4, Corollary 1]{Brion}, and by Thibon  by means of vertex operators \cite[\S 3]{Thibon}. More recent proofs have been obtained by interpreting the Kronecker coefficients in the setting of representations of partition algebras \cite{BDO}, and by constructing appropriate frameworks for addressing stability in general \cite{ChurchEllenbergFarb:FI,SamSnowdenGeneralizedStability}.

 The stable value of the sequence $\K{\lambda+(m)}{\mu+(m)}{\nu+(m)}$ does not depend on the first part of $\lambda$, $\mu$ and $\nu$. Accordingly, it will be denoted $\RK{\cut{\lambda}}{\cut{\mu}}{\cut{\nu}}$,
and  called here a \emph{reduced Kronecker coefficient}.  More precisely, the
sequence with general term $g_{\cut{\lambda}[N], \cut{\mu}[N], \cut{\nu}[N]}$, beginning
at some suitably large $N$, has a limit whose value we label with
$\cut{g}_{\cut{\lambda}, \cut{\mu}, \cut{\nu}}$.  Thus, while the reduced
Kronecker coefficients are defined for any triple of partitions, the Kronecker
coefficients are only defined for triples of partitions of the same integer.

M. Brion has shown that the sequence of Kronecker coefficients $\K{\lambda+(m)}{\mu+(m)}{\nu+(m)}$ is weakly increasing \cite[\S 3.4, Corollary 1]{Brion}. This implies in particular that 
$\K{\lambda}{\mu}{\nu} \leq \RK{\cut{\lambda}}{\cut{\mu}}{\cut{\nu}}$

It is shown in \cite[Theorem 1.5]{BOR-JA} that $\K{\lambda+(m)}{\mu+(m)}{\nu+(m)}=\RK{\cut{\lambda}}{\cut{\mu}}{\cut{\nu}}$ holds for all $m$ such that $|\lambda|+m \geq N_0(\alpha, \beta, \gamma)$, where
\begin{equation}\label{N0}
N_0(\alpha, \beta, \gamma)=\frac{|\alpha|+\alpha_1+|\beta|+\beta_1+|\gamma|+\gamma_1}{2}.
\end{equation}

 Moreover, Murnaghan showed that the reduced Kronecker coefficients are zero unless the following inequalities hold:
\begin{lemma}[Murnaghan's inequalities]  The reduced Kronecker coefficient  $\RK{{\lambda}}{{\mu}}{{\nu}}$ are zero unless the following three conditions hold:
\begin{align}\label{Murnaghan inq}
\left\lbrace
\begin{matrix}
|\lambda| \leq |\mu| + |\nu|\\
|\mu| \leq |\lambda| + |\nu|\\
|\nu| \leq |\lambda| +  |\mu| 
\end{matrix}
\right .
\end{align}

\end{lemma}

\subsection{Brion's formula and the generating series for the Reduced Kronecker Coefficients.}
In \cite[\S 3.4, Corollary 1]{Brion}, M. Brion obtained the following formula 
 for the reduced Kronecker coefficients.
\begin{proposition}\label{P:Brion Formula}
For any three partitions $\alpha$, $\beta$ and $\gamma$,
\begin{equation}\label{stable:1}
\RK{\alpha}{\beta}{\gamma}=\scalar{s_{\alpha}[X] s_{\beta}[Y]}{\sigma[XY]s_{\gamma}[XY+X+Y]}_{X,Y}.
\end{equation}
\end{proposition}

 We include an elementary  proof of Brion's formula (and, at the same time, of
 Murnaghan's stability), based on the following elementary lemma.

\begin{lemma}\label{lemma:stable}
A sequence with general term $u_n$ stabilizes (is eventually constant) if and only if its generating series $g(t)=\sum_n u_n t^n$ takes the form $P(t)/(1-t)$ with $P(t)$ a polynomial. 
Moreover, the stable value of the sequence is $P(1)$. 
\end{lemma}

\begin{proof}[Proof of Proposition \ref{P:Brion Formula}]

All scalar products appearing in this proof will be taken with respect to $X,Y$,
as in the statement of the proposition.

Let $\alpha$, $\beta$ and $\gamma$ be three partitions.  We have, for $n$ big enough,
\[
\K{ \alpha[n]}{\beta[n]}{\gamma[n]}=\scalar{s_{\gamma[n]}[XY]}{s_{ \alpha[n]}[X]s_{\beta[n]}[Y]}.
\]
We can write as well
\[
\K{\alpha[n]}{\beta[n]}{\gamma[n]}=\scalar{s_{\gamma[n]}[XY]}{\sum_{a} s_{\alpha[a]}[X] \cdot \sum_b s_{\beta[b]}[Y]}.
\]
Indeed, the extra terms in the right--hand side of the scalar product do not contribute since they do not the same degree as the left--hand side. We simplify the series on the right--hand side using \eqref{V1}, to get
\[
\K{\alpha[n]}{\beta[n]}{\gamma[n]}=\scalar{s_{\gamma[n]}[XY]}{\sigma[X] s_{\alpha}[X-1] \cdot \sigma[Y] s_{\beta}[Y-1]}.
\]
Let us introduce the generating series
\begin{align*}
g(t)
&=\sum_n \scalar{s_{\gamma[n]}[XY]}{\sigma[X] s_{\alpha}[X-1]  \sigma[Y] s_{\beta}[Y-1]} t^n\\
&=\scalar{\sum_n s_{\gamma[n]}[XY] t^n}{\sigma[X] s_{\alpha}[X-1] \sigma[Y] s_{\beta}[Y-1]}.
\end{align*}
From \eqref{V2}, with $XY$ instead of $X$, we have
\[
g(t)=\scalar{\sigma[tXY] s_{\gamma}[tXY-1] }{\sigma[X] s_{\alpha}[X-1] \cdot \sigma[Y] s_{\beta}[Y-1]}
\]
Using the adjoints of $\sigma[X]$ and $\sigma[Y]$ (see Lemma \ref{L1}), we get 
\begin{align*}
g(t)=\scalar{\sigma[t(X+1)(Y+1)] s_{\gamma}[t(X+1)(Y+1)-1] }{s_{\alpha}[X-1]  s_{\beta}[Y-1]}\\
=\sigma[t] \scalar{\sigma[tXY]\sigma[tX]\sigma[tY] s_{\gamma}[t(X+1)(Y+1)-1] }{s_{\alpha}[X-1] s_{\beta}[Y-1]}\\
=\sigma[t] \scalar{\sigma[tXY] s_{\gamma}[t(X+1)(Y+1)-1] }{s_{\alpha}[X+t-1] s_{\beta}[Y+t-1]}.
\end{align*}
(Note that we specialized $\sigma[AX^\perp]$ to $A=1$.)
That is, $g(t)= \frac{1}{1-t} P(t)$, with $P(t)$ equal to
\[ P(t)=\scalar{\sigma[tXY] s_{\gamma}[t(X+1)(Y+1)-1] }{s_{\alpha}[X+t-1]  s_{\beta}[Y+t-1]}.\]
We expand $\sigma[tXY]=\sum_{k=0}^{\infty} h_k[XY] t^k$, and observe that the terms $h_k[XY] t^k$, for $k$ big enough, do not contribute to the scalar product. Indeed, they are homogeneous of total degree $2k$ in $X$ and $Y$, while the right--hand side has degree $|\alpha|+|\beta|$.
The infinite series can thus be truncated, and,  $P(t)$ is equal to
\[
\scalar{\sum_{k=0}^{k_0} h_k[XY] t^k s_{\gamma}[t(X+1)(Y+1)-1] }{s_{\alpha}[X+t-1] \cdot s_{\beta}[Y+t-1]}.\] 
Under this form, it is manifest that $P(t)$ is a polynomial in $t$.

After Lemma \ref{lemma:stable}, the sequence of coefficients of $g(t)$ is eventually constant. This sequence of coefficients coincides with the sequence of Kronecker coefficients $\K{\alpha[n]}{\beta[n]}{\gamma[n]}$ for $n \gg 0$. This proves that this sequence of Kronecker coefficients is eventually constant. 
Finally, substituting $1$ for $t$ in the expression for $P(t)$ we get Brion's Formula.
\end{proof}

Brion's formula is equivalent to the following identity:
\[
\sigma[XY]s_\gamma[XY+X+Y]=   \sum_{\alpha,\beta}\RK{\alpha}{\beta}{\gamma}s_{\alpha}[X]s_{\beta}[Y].
\]
Introducing a third alphabet $Z$, multiplying by $s_{\gamma}[Z]$, and summing over all partitions $\gamma$, yields the following analogue of Cauchy's formula for reduced Kronecker coefficients, which makes manifest the symmetry in the three indexing partitions:
\begin{equation}\label{symmetric Brion formula}
\sum_{\alpha,\beta,\gamma}\RK{\alpha}{\beta}{\gamma}s_{\alpha}[X]s_{\beta}[Y]s_{\gamma}[Z]
= \sigma[XYZ+XY+XZ+YZ] .
\end{equation}

Conversely, Brion's Formula is obtained from \eqref{symmetric Brion formula} by taking the scalar product with $s_{\gamma}[Z]$ and making use of the reproducing kernel property for $\sigma[AZ]$ with $A=XY+X+Y$.


\subsection{Reduced Kronecker coefficients indexed by three one--row shapes or three one--column shape}

Let $x$, $y$ and $z$ be three variables. By specializing, in the generating series
for the reduced Kronecker coefficients \eqref{symmetric Brion formula}, the
alphabets $X$, $Y$ and $Z$ to $x$, $y$, and $z$, we get that
\[\sigma[xyz+xy+xz+yz] = \sum_{(a,b,c) \in \NN^3} \RK{(a)}{(b)}{(c)} x^a y^b
z^c,\]
which is the ordinary generating function for the reduced Kronecker coefficients indexed by three one--row shapes.

Similarly, we get the generating function for the reduced Kronecker coefficients indexed by three one--column shapes by specializing, in  \eqref{symmetric Brion formula}, the alphabets $X$, $Y$ and $Z$ to $-\varepsilon x$, $-\varepsilon  y$, and $-\varepsilon  z$:
\[
\sigma[-\varepsilon xyz+xy+xz+yz] = 
\sum_{(a,b,c) \in \NN^3} \RK{(1^a)}{(1^b)}{(1^c)} x^a y^b z^c.
\]

From the properties of the series $\sigma$, one gets straightforwardly the following simple expressions for the generating series:
\[
\sigma[-\varepsilon xyz+xy+xz+yz] =  \frac{1+xyz}{(1-xy)(1-xz)(1-yz)}
\]
and 
\[
\sigma[xyz+xy+xz+yz] =  \frac{1}{(1-xyz)(1-xy)(1-xz)(1-yz)},
\]
which is, as an aside,
\[
\frac{1}{1-(xyz)^2} \cdot \sigma[-\varepsilon xyz+xy+xz+yz].
\]

\begin{proposition}\label{files and columns}
The above generating series admit the following expansions:
\begin{equation}
\sigma[-\varepsilon xyz + xy + xz + yz]=
 \sum_{(a,b,c) \in \mathcal{C} \cap \NN^3} x^a y^b z^c  \label{gseries col expanded}
\end{equation}
and 
\begin{equation}
\sigma[xyz + xy + xz + yz]=
 \sum_{(a,b,c) \in \mathcal{C} \cap \NN^3} \left(1+\left[ \frac{\min \{\ell_1,\ell_2,\ell_3\}}{2} \right]\right) x^a y^b z^c \label{gseries row expanded}
\end{equation}
where $\ell_i=\ell_i(a,b,c)$ with
\begin{equation}\label{elli}
\begin{matrix}
\ell_1(a,b,c) = b+c-a ,\\ 
\ell_2(a,b,c) = a+c-b ,\\ 
\ell_3(a,b,c) = a+b-c\\ 
\end{matrix}
\end{equation}
and $\mathcal{C}$ is the cone in $\RR^3$ with coordinates $a$, $b$, $c$, defined by  
\begin{equation}\label{cone inequalities}
\begin{matrix}
\ell_1 \geq 0,\\ 
\ell_2 \geq 0,\\
\ell_3 \geq 0.
\end{matrix}
\end{equation}
\end{proposition}

\begin{remark}
The inequalities \eqref{cone inequalities}, defining the cone $\mathcal{C}$, are
precisely those described by Murnaghan's inequalities \eqref{Murnaghan inq}, as a non--vanishing condition for the reduced Kronecker coefficients. 

\end{remark}

\begin{remark}\label{ex_3cols}
The proposition amounts to explicit formulas for the reduced Kronecker coefficients indexed by three one--row shapes and three one--column shapes:
\begin{align*}
\RK{(1^a)}{(1^b)}{(1^c)} &=
\left\lbrace
\begin{matrix}
1 & \text{if $(a,b,c) \in \mathcal{C}$},\\
0 & \text{ otherwise.}
\end{matrix}
\right. \\
\RK{(a)}{(b)}{(c)} &=
\left\lbrace
\begin{matrix}
1+\left[ \frac{\min \{\ell_1,\ell_2,\ell_3\}}{2} \right] & \text{if $(a,b,c) \in \mathcal{C}$},\\
0 & \text{ otherwise.}
\end{matrix}
\right.
\end{align*}

They can be derived from \cite[Corollary 5]{Rosas} and \cite[Theorem 13]{Rosas}
respectively.  We give a different proof below.
\end{remark}

\begin{proof}
We first prove \eqref{gseries col expanded}. We begin with the expansion 
\[
\frac{1}{(1-xy)(1-xz)(1-yz)}=\sum_{(i,j,k)\in\NN^3} (xy)^i (xz)^j (yz)^k = \sum_{(i,j,k) \in \NN^3} x^{i+j} y^{i+k}z^{j+k}.
\]
We solve the following system in $i$, $j$, $k$, over the rational numbers:
\[
\begin{matrix}
a &=& i &+& j & & \\
b &=& i & &   &+& k \\
c &=&   & & j &+& k   
\end{matrix}
\]
It has a unique solution, $i=(a+b-c)/2$, $j=(a+c-b)/2$, $k=(b+c-a)/2$. This solution is a triple of nonnegative integers if and only if $a+b+c \equiv 0 \mod 2$, and the inequalities \eqref{cone inequalities} hold. Therefore,
\[
\frac{1}{(1-xy)(1-xz)(1-yz)}=\sum x^a y^b z^c,
\]
where the sum is over all $(a,b,c) \in \NN^3 \cap \mathcal{C}$ that satisfy $a+b+c \equiv 0 \mod 2$. 
As a consequence, we have also
\[
\frac{xyz}{(1-xy)(1-xz)(1-yz)}=\sum x^a y^b z^c
\]
where the sum is over all $(a,b,c) \in \NN^3 \cap \mathcal{C}$ that satisfy $a+b+c \equiv 1 \mod 2$. Formula \eqref{gseries col expanded} follows.

Let us prove now  \eqref{gseries row expanded}. We observe that 
\begin{align*}
\sigma[xyz+xy+xz+yz]
&=&\frac{1}{1-(xyz)^2} \sigma[-\varepsilon xyz+xy+xz+yz]\\
&=&\sum_{m \in \NN} x^{2m}y^{2m}z^{2m} \cdot \sum_{(i,j,k) \in \NN^3 \cap \mathcal{C} } x^i y^j z^k.  \end{align*}
Therefore, the coefficient $\RK{(a)}{(b)}{(c)}$ of $x^a y^b z^c$ in $\sigma[xyz+xy+xz+yz]$ is the number of integers $m \geq 0$ such that $(a,b,c)-2(m,m,m) \in \mathcal{C}$. It is obtained as the number of solutions $m \geq 0$ of
\[
 \forall i \in\{ 1,2,3\}, \quad \ell_i((a,b,c)-2(m,m,m)) \geq 0
\]
But $\ell_i((a,b,c)-2(m,m,m))=\ell_i(a,b,c)-2\ell_i(m,m,m)=\ell_i(a,b,c)-2 m$.
Therefore $\RK{(a)}{(b)}{(c)}$ is the number of integers $m \geq 0$ such that $m
\leq \ell_i(a,b,c)/2$ for all $i=1,2,3$. This is $1+[\min_i \ell_i/2]$, as
claimed.
\end{proof}

\section{A factorization.} \label{a-factorization}

Given an alphabet $X'$, define   $\vertexop{X'}{X}=\sigma[X' X] \sigma\left[-\frac{1}{X'} X^{\perp}\right]$. This allows us to consider the two vertex operators we are working with in this paper simultaneously, as for $X'=t$ we recover the standard vertex operator, and for  $X'=-\epsilon t$ the vertex operator for columns.

Let $X$, $Y$, $Z$, $X'$, $Y'$ and $Z'$ be six independent alphabets. Let $F(X,Y,Z)=XYZ+XZ+YZ+XY$.
For any triple of partitions $\alpha$, $\beta$, $\gamma$,  let $\Phi_{\alpha,\beta,\gamma}$ be the series 
\begin{equation}\label{def Phi}
\Phi_{\alpha,\beta,\gamma}=
\scalar{\sigma[F(X,Y,Z)]}{\vertexop{X'}{X} s_{\alpha}[X] \vertexop{Y'}{Y} s_{\beta}[Y] \vertexop{Z'}{Z} s_{\gamma}[Z]}_{X,Y,Z}
\end{equation}
Then $\Phi_{\alpha,\beta,\gamma}$ is a symmetric series in each of the sets of
variables $X'$, $Y'$ and $Z'$.   Lemma \ref{main lemma} below decribes this
symmetric series in more detail.

\begin{lemma}\label{main lemma}
For any partitions $\alpha$, $\beta$ and $\gamma$, 
there exists a symmetric function $Q_{\alpha,\beta,\gamma}$ (in the alphabets $X'$, $Y'$ and $Z'$) such that 
\[
\Phi_{\alpha,\beta,\gamma}=\sigma[X'Y'Z'+X'Y'+X'Z'+Y'Z'] \cdot Q_{\alpha,\beta,\gamma}.
\]
The symmetric function $Q_{\alpha,\beta,\gamma}$  is the coefficient of $s_{\alpha}[X] s_{\beta}[Y] s_{\gamma}[Z]$ in the expansion in the Schur basis of $\sigma[H]$ (as a symmetric series in $X$, $Y$ and $Z$), where
\begin{multline*}
H =(X+X')(Y+Y')(Z+Z')+(X+X')(Y+Y')+(X+X')(Z+Z')\\+(Y+Y')(Z+Z')
-(X'Y'Z'+X'Y'+X'Z'+Y'Z') \\-{X}/{X'}-{Y}/{Y'}-{Z}/{Z'}.
\end{multline*}
\end{lemma}
The main point of this lemma is that  $Q_{\alpha,\beta,\gamma}$ is not just a symmetric series but a symmetric function;  it has finitely many non--zero homogeneous components.

\begin{proof} Fix three partitions $\alpha$, $\beta$ and $\gamma$. 

In \eqref{def Phi}, we move the $\vertexop{}{}$ to the left--hand side of the scalar product by taking adjoints. 
The adjoint of   $\vertexop{X'}{X}$ (with respect to the alphabet $X$) is the
operator $\sigma[-\frac{1}{X'} X] \sigma[X' X^{\perp}]$ that sends $f[X]$ to
$\sigma[-\frac{1}{X'} X] f[X+X']$ and, likewise, for $\vertexop{Y'}{Y}$ and
$\vertexop{Z'}{Z}$.
As a result, $\Phi_{\alpha,\beta,\gamma}=\scalar{\sigma[G]}{s_{\alpha}[X]s_{\beta}[Y]s_{\gamma}[Z]}$ with 
\[
G=F(X+X',Y+Y',Z+Z')-{X}/{X'}-{Y}/{Y'}-{Z}/{Z'}.
\]
Let us split $G$ as $G=F(X',Y',Z')+H$. That is, $H$ is obtained from $G$ by deleting all monomials that do not involve 
  $X$, $Y$ nor $Z$. Then $H$ is given by the formula in the lemma. 
 
  We have $\sigma[G]=\sigma[F(X',Y',Z')] \cdot \sigma[H]$ by Lemma \ref{L1}. Since $\sigma[F(X',Y',Z')]$ does not depend on $X$, $Y$ and $Z$, it can be factored out of the scalar product:
\[
\Phi_{\alpha,\beta,\gamma}=\sigma[F(X',Y',Z')]  \cdot \scalar{\sigma[H]}{s_{\alpha}[X] s_{\beta}[Y] s_{\gamma}[Z]}_{X,Y,Z}.
\] 
This gives the announced factorization, since the scalar product in the above formula is equal to $Q_{\alpha,\beta,\gamma}$.

We contend that the non--zero homogeneous components of $Q_{\alpha,\beta,\gamma}$ (in the variables in $X'$, $Y'$ and $Z'$) have bounded degrees. Indeed, we have the expansion
\[
\sigma[H]=\sum_{k=0}^{\infty} h_k[H]. 
\]
But all terms in $H$ have total degree at least $1$, with respect to the variables $X$, $Y$ and $Z$. Therefore for each $k$, $h_k[H]$ is a sum of homogeneous symmetric functions (in $X$, $Y$ and $Z$) of total degrees $\geq k$. When $k >  |\alpha|+|\beta|+|\gamma|$, the term $h_k[H]$ does not contribute to the scalar product with $s_{\alpha}[X] s_{\beta}[Y] s_{\gamma}[Z]$. The sum can therefore be truncated, so that
\[
Q_{\alpha,\beta,\gamma}=\scalar{\sum_{k=0}^{|\alpha|+|\beta|+|\gamma|} h_k[H]}{s_{\alpha}[X] s_{\beta}[Y] s_{\gamma}[Z]}_{X,Y,Z}.
\]
This makes clear that the homogeneous components of $Q$, as a symmetric series in $X'$, $Y'$ and $Z'$, have bounded degree. 
\end{proof}



\section{Hook stability}\label{hook-stability}

In this section, we show the existence of a stability phenomenon reminiscent of
the one described by Murnaghan, when we simultaneously increase the first row and first column of each of the three indexing partitions of a Kronecker coefficient. We call this ``hook stability''.

We start by establishing the corresponding property for reduced Kronecker coefficients (Section \ref{hook reduced}). The hook stability property for Kronecker coefficients is then deduced in Section \ref{hook subsection}. An alternative approach to proving this property, using only well--known properties of Kronecker coefficients, is explored in Section \ref{alternative subsection}.

\subsection{Stability for reduced Kronecker coefficients under first column increasing.}\label{hook reduced}

In this section we show that reduced Kronecker coefficients \emph{themselves} stabilize when we increase the first column of each of their three indexing partitions.

\begin{example}
  The reduced Kronecker coefficients $\RK{(2,2) \cup 1^{k}}{(3) \cup 1^{k}}{(4) \cup 1^{k}}$ stabilize with stable value $204$. Their first values for $k=1$,$2$,\ldots are   
 \[1, 17, 66, 133, 180, 198, 203, 204, 204, 204 \ldots
 \]
 and $\RK{(2,2) \cup 1^{k}}{(3) \cup 1^{k}}{(4) \cup 1^{k}}=204$ for all $k \geq 8$. 
\end{example}

More interesting examples are given in Remark \ref{ex_3cols} (where we combinatorially describe all possible
situations that can  obtained when we start with three empty shapes) and Example \ref{ex:emptypart}.

 We now proceed to study the general situation.
Let $\alpha$, $\beta$ and $\gamma$ be three partitions. 

 \begin{theorem}\label{3columns}
 For any triple of partitions $\alpha$, $\beta$, $\gamma$, there exist integers $k_1$, $k_2$, $k_3$ and $\SSK{\alpha}{\beta}{\gamma}$ such that whenever $a \geq \ell(\alpha)$, $b \geq \ell(\beta)$, $c \geq \ell(\gamma)$, and 
 \begin{equation}\label{trans-cone}
\begin{matrix}
b+c-a &\geq& k_1\\
a+c-b &\geq& k_2\\
a+b-c &\geq& k_3,
\end{matrix}
\end{equation}
we have $\RK{\alpha +  (1^a)}{\beta + (1^b)}{\gamma + (1^c)}=\SSK{\alpha}{\beta}{\gamma}$.
\end{theorem}
In light of Theorem \ref{3columns}, we call the value
$\SSK{\alpha}{\beta}{\gamma}$ the \emph{column stable value} of the reduced
Kronecker coefficient.

The conditions $a \geq \ell(\alpha)$, $b \geq \ell(\beta)$, $c \geq
\ell(\gamma)$ ensure us that, after adding cells to the new  first columns of
the three original partitions, we obtain proper partitions. Note that they define a translation of the cone described by Murnaghan's inequalities: see (\ref{Murnaghan inq}).

\begin{proof}
For any nonnegative integers $a$, $b$, $c$, set
\[
\scol_{a,b,c}=\scalar{\sigma[F(X,Y,Z)]}{\tilde{s}_{(1^a|\alpha)}[X] \tilde{s}_{(1^b|\beta)}[Y] \tilde{s}_{(1^c|\gamma)}[Z]}.
\]
with $F(X,Y,Z)=XYZ+XY+XZ+YZ$.
Comparing with \eqref{symmetric Brion formula} we obtain that, when  $a \ge \ell(\alpha)$, $b \ge \ell(\beta)$ and $c \ge \ell(\gamma)$, 
\[\scol_{a,b,c}=\RK{\alpha + (1^a)}{\beta + (1^b)}{\gamma + (1^c)}.
\]
Let us consider 
the generating series
$
\Scol_{\alpha,\beta,\gamma}=\sum_{a,b,c} \scol_{a,b,c} x^a y^b z^c.
$
Then
\[
\Scol_{\alpha,\beta,\gamma} = \langle \sigma[F(X,Y,Z)] | \vertexop{-\varepsilon x}{X} s_{\alpha}[X] \vertexop{-\varepsilon y}{Y} s_{\beta}[Y]\vertexop{-\varepsilon z}{Z} s_{\gamma}[Z] \rangle
\]
That is, $\Scol_{\alpha,\beta,\gamma}$ is the specialization of the series $\Phi_{\alpha,\beta,\gamma}$ of Lemma \ref{main lemma} at $X'=-\varepsilon x$, $Y'=-\varepsilon y$ and $Z'=-\varepsilon z$. Let $\polcol_{\alpha,\beta,\gamma}(x,y,z)$ be the polynomial obtained from $Q_{\alpha,\beta,\gamma}$ by means of the same specialization. After Lemma \ref{main lemma}, we have thus
$
\Scol_{\alpha,\beta,\gamma}=\sigma[-\varepsilon xyz+xy+xz+yz] \cdot \polcol_{\alpha,\beta,\gamma}.
$
Set $\vvar^{(a,b,c)}=x^a y^b z^c$.
Let $\sum_{\omega \in \Omega} q_{\omega} \vvar^{\omega}=
\polcol_{\alpha,\beta,\gamma}$ be the expansion of
$\polcol_{\alpha,\beta,\gamma}$ in monomials, where $\Omega$ is the (finite)
support of $\polcol_{\alpha,\beta,\gamma}$. It follows from \eqref{gseries col
expanded} that
\[
\sigma[-\varepsilon xyz+xy+xz+yz] =
\sum_{\theta \in \mathcal{C} \cap \NN^3} \vvar^{\theta}.
\]
Therefore, \[
\Scol_{\alpha,\beta,\gamma}(x,y,z)=\sum_{\omega \in \Omega, \theta \in \mathcal{C} \cap \NN^3} q_{\omega} \vvar^{\omega+\theta}.
\]
It follows that, for any $\tau=(a,b,c) \in \ZZ^3$, $\scol_{\tau}=\sum_{\omega} q_{\omega}$, 
where the sum is over all $\omega$ such that $\tau-\omega \in \mathcal{C}$.
Recall that the cone $\mathcal{C}$ is defined by the inequalities $\ell_i \geq 0$ (see Proposition \ref{files and columns}). Therefore, the sum is over all $\omega$ such that $\ell_i(\tau-\omega) \geq 0$ for all $i$, or, equivalently, $\ell_i(\tau) \geq \ell_i(\omega)$ for all $i$.

Suppose now that  $\ell_i(\tau) \geq \ell_i(\omega)$ for all $i$ and all $\omega \in \Omega$, or, equivalently, that 
$\ell_i(\tau) \geq \max_{\omega \in \Omega} \ell_i(\omega)$ for all $i$. Then  $\scol_{\tau}=\sum_{\omega \in \Omega} q_{\omega}$, a value that does not depend on $\tau$.

This proves the theorem, with $k_i=\max_{\omega \in \Omega} \ell_i(\omega)$.
\end{proof}

\begin{remark}\label{rem:ki}
One can show that in Theorem \ref{3columns}, one can take
\[
\begin{array}{rcl}
k_1 &=& |\alpha|+\alpha_1+\beta'_1+\gamma'_1, \\
k_2 &=& |\beta|+\beta_1+\alpha'_1+\gamma'_1,   \\
k_3 &=& |\gamma|+\gamma_1+\alpha'_1+\beta'_1.
\end{array}
\]
\end{remark}

\begin{example}\label{ex:emptypart}
Let us compute some polynomials $\polcol_{\alpha,\beta,\gamma}$.

We will consider the case when $\alpha$, $\beta$ and $\gamma$ are one--row
shapes,  $(p)$, $(q)$ and $(r)$ respectively. The coefficient of $s_{(p)}[X]
s_{(q)}[Y] s_{(r)}[Z]$ in $\sigma[H(-\varepsilon x,-\varepsilon y,-\varepsilon
z)]$ can be obtained by specializing the alphabets to only one letter:  $X=\{x_1\}$, $Y=\{y_1\}$, $Z=\{z_1\}$, and taking the coefficient of $x_1^p y_1^p z_1^r$. That is, the generating function $\sigma[H(-\varepsilon x,-\varepsilon y,-\varepsilon z)]$ becomes an ordinary generating function:
\begin{multline*}
\sum \polcol_{(p ),(q),(r )}x_1^p y_1^q z_1^r =
\frac{(1+x y z_1)(1+x z y_1)(1+ y z x_1) }{(1-x_1 y_1 z_1)} \\
\times \frac{(1+y x_1)(1+x y_1) (1+z x_1)(1+x z_1)(1+ z y_1)(1+ y z_1)}{(1-x_1 y_1)(1-x_1 z_1)(1-y_1 z_1)(1+\frac{x_1}{x})(1+\frac{y_1}{y})(1+\frac{z_1}{z})}
\end{multline*}

From this, it follows, for instance,
\begin{align*}
&\polcol_{\ep,\ep,\ep}=1,\\
&\polcol_{\ep, \ep,(1)}=x+y+xy-1/z,\\
&\polcol_{\ep, \ep,(2)} =x^2 y^2 + x^2 y + x y^2 + x y - x y/z -x/z-y/z+1/z^2,\\
&\polcol_{\ep,(1),(1)}      =x^2 y z + x^2 y +x^2 z + 2\, x y z +x^2,\\
                         & \phantom{xxxxxxxxxxxx}+ x y + xz + y z -x -x/y- x/z +1/(yz)-1, 
\end{align*}

Let us consider more closely the case $\ep$, $\ep$, $(1)$. This case corresponds to the reduced Kronecker coefficients $\RK{(1^a)}{(1^b)}{(2,1^{c-1})}$. From the description $\RK{(1^a)}{(1^b)}{(2,1^{c-1})}=\sum q_{\omega}$, with the sum over the $\omega$ in the support of $\polcol$ such that $(a,b,c) \in \omega + \mathcal{C}$, we obtain the following explicit description (it is assumed that $c \geq 1$):
\[
\RK{(1^a)}{(1^b)}{(2,1^{c-1})}=
\left\lbrace
\begin{matrix}
1 & \text{ for $c= |a-b|$ with $a+b > c+1$}\\
  & \text{ and for $c > |a-b|$ with $a+b=c+1$,}\\
2 & \text{ for $c> |a-b|$ with $a+b > c+1$,}\\
0 & \text{ otherwise.}
\end{matrix}
\right.
\]
This is  the Kronecker coefficient $\K{(n-a,1^a)}{(n-b,1^b)}{(n-c-1,2,1^{c-1})}$ for $n \geq (a+b+c+5)/2$. 

This result also follows from the computations in \cite{Rosas} and \cite{Thibon}.
\end{example}

\subsection{Towards hook stability for the Kronecker coefficients.}\label{hook subsection}

We  discuss how, combining our results, with the classical stability phenomena of Murnaghan, we obtain that the Kronecker coefficients are
stable when we increase the first row and first column of the three indexing
partitions \emph{simultaneously}. 
We will be using the notations for $\cut{\lambda}$, $\cutt{\lambda}$, $\lambda\op{a}{b}$ as defined in Section \ref{partitions}.

\begin{example}
Table \ref{table:ex KRKSSK} presents the Kronecker coefficients $\K{\lambda\op{i}{j}}{\lambda\op{i}{j}}{\lambda\op{i}{j}}$ for $\lambda=(3,3)$ and $i$ and $j$ between $0$ and $9$. 
We know that each column of the table is stable because of Murnaghan's result,
and that each row is eventually zero because these sequences will eventually
fail a condition for positivity described by Dvir, Klemm, and Clausen--Meier in \cite{Dvir, Klemm, ClausenMeier}. But we observe a more general stability phenomenon. There is a grey region where the coefficients are  $145$. 

\begin{table}[htbp]
\begin{tabular}{c|rrrrrrrrrr}
\backslashbox{$i$}{$j$}&0 & 1 & 2 & 3 & 4 & 5 & 6 & 7 & 8 & 9\\
\hline
0 & 0 & 1 & 5 & 5 & 1 & 0 & 0 & 0 & 0 & 0\\
1 & 1 & 8 & 27 & 40 & 30 & 11 & 1 & 0 & 0 & 0\\
2 & 1 & 15 & 53 & 89 & 91 & 64 & 33 & 11 & 1 & 0\\
3 & 2 & 19 & 62 & 108 & 129 & 122 & 97 & 64 & 33 & 11\\
4 & 2 & 19 & 63 & 112 & 138 & 141 & 135 & 122 & 97 & 64\\
5 & 2 & 19 & 63 & 112 & 139 & \cbox & 144 & 141 & 135 & 122\\
6 & 2 & 19 & 63 & 112 & 139 & \cbox & \cbox & \cbox & 144 & 141\\
7 & 2 & 19 & 63 & 112 & 139 & \cbox & \cbox & \cbox & \cbox & \cbox\\
8 & 2 & 19 & 63 & 112 & 139 & \cbox & \cbox & \cbox & \cbox & \cbox\\
9 & 2 & 19 & 63 & 112 & 139 & \cbox & \cbox & \cbox & \cbox & \cbox
\end{tabular}
\caption{The Kronecker coefficients $\K{(3,3)\op{i}{j}}{(3,3)\op{i}{j}}{(3,3)\op{i}{j}}$. }\label{table:ex KRKSSK}
\end{table}

\end{example}

Let $\lambda$, $\mu$ and $\nu$ be three non--empty partitions of the same
weight. Let $a$, $b$, $c$ and $m$ be nonnegative integers, such that $a$, $b$
and $c$ do not exceed $m$. Under certain conditions, made precise in Theorem
\ref{thm:hook stab} below, we will have:
\[
\K{\lambda\op{m-a}{a}}{\mu\op{m-b}{b}}{\nu\op{m-c}{c}} = \RK{\cut{\lambda} \cup (1^a)}{\cut{\mu} \cup (1^b)}{\cut{\nu} \cup (1^c)} = \SSK{\cutt{\lambda}}{\cutt{\mu}}{\cutt{\nu}}.
\]
This is made precise in the following theorem. 

\begin{theorem}\label{thm:hook stab}
For any triple of non--empty partitions $\lambda$, $\mu$, $\nu$ of the same weight, there exists integers $d_1$, $d_2$, $d_3$ and $d$ such that for all $(a,b,c,m) \in \NN^4$ with 
\begin{equation}\label{translation of E}
\begin{matrix}
\ell_i(a,b,c) &\geq  d_i & \text{ for all } i\in\{1,2,3\},\\
m -(a+b+c)/2  &\geq  d,     & \\
m \geq a, b, c.
\end{matrix}
\end{equation}
we have, 
\begin{equation}\label{KRKSSK}
\K{\lambda\op{m-a}{a}}{\mu\op{m-b}{b}}{\nu\op{m-c}{c}} = \RK{\cut{\lambda} \cup (1^a)}{\cut{\mu} \cup (1^b)}{\cut{\nu} \cup (1^c)} = \SSK{\cutt{\lambda}}{\cutt{\mu}}{\cutt{\nu}}.
\end{equation}
\end{theorem}
The linear forms $\ell_i(a,b,c)$ in the theorem are those defined in \eqref{elli}.

\begin{proof}

Let $N$ be the weight of $\lambda$, $\mu$ and $\nu$. The second equality in \eqref{KRKSSK} holds when $\ell_i(a+\lambda'_1-1,b+\mu'_1-1,c+\nu'_1-1) \geq k_i(\cutt{\lambda},\cutt{\mu},\cutt{\nu})$ for all $i$, where $k_i$ are defined in Theorem \ref{3columns}. We have $\ell_i(a+\lambda'_1-1,b+\mu'_1-1,c+\nu'_1-1)=\ell_i(a,b,c)+\ell_i(\lambda'_1,\mu'_1,\nu'_1)-1$. Therefore the second equality holds when, for all $i$, we have $\ell_i(a,b,c) \geq d_i$, with $d_i=k_i(\cutt{\lambda},\cutt{\mu},\cutt{\nu})-\ell_i(\lambda'_1,\mu'_1,\nu'_1)+1$.

On the other hand, the first equality in \eqref{KRKSSK} holds when $m+N \geq N_0(\cut{\lambda} \cup (1^a),\cut{\mu} \cup(1^b),\cut{\nu} \cup (1^c))$ (the number $N_0$ as defined in \eqref{N0}). Lemma \ref{lemma:N0}, that comes just below, shows that
\[
N_0(\cut{\lambda} \cup (1^a),\cut{\mu} \cup(1^b),\cut{\nu} \cup (1^c)) \leq N_0(\cutt{\lambda}, \cutt{\mu}, \cutt{\nu}) + \frac{\lambda'_1+\mu'_1+\nu'_1}{2}+\frac{a+b+c}{2}.
\]
From this we conclude that the first equality holds when $m -(a+b+c)/2  \geq  d$ with 
$
d=N_0(\cutt{\lambda}, \cutt{\mu}, \cutt{\nu}) + \frac{\lambda'_1+\mu'_1+\nu'_1}{2} - N.
$
\end{proof}

\begin{example}
Let us go back to Table 1. The reduced Kronecker coefficients $\RK{(3) \cup (1^j)}{(3) \cup (1^j)}{(3) \cup (1^j)}$ are  $2, 19, 63, 112, 139$ and then, for $j \geq 5$, to  $\SSK{(2)}{(2)}{(2)}=145.$ 
Moreover, the sequences are stable  when $j \geq 5$ and $i-5 \geq (j-5)/2$.

 The  Kronecker coefficients  of the main diagonal
 are $0$, $8$, $53$, $108$, $138$, and finally $145$ for all $m \geq 5$. The values of the bounds $d$ and $d_i$ for the stability degrees  given in the proof of Theorem \ref{thm:hook stab} are $d=3$ and $d_i=5$. This corresponds to stability for $j \geq 5$ and $i-j/2 \geq 3$, which is not far from being sharp.

\end{example}

\begin{lemma}\label{lemma:N0}
Let $\lambda$, $\mu$ and $\nu$ be three non--empty partitions with the same weight. We have
$
N_0(\cut{\lambda}, \cut{\mu}, \cut{\nu}) \leq N_0(\cutt{\lambda}, \cutt{\mu}, \cutt{\nu}) + \frac{\lambda'_1+\mu'_1+\nu'_1}{2}.
$
\end{lemma}

\begin{proof}
Recall from \eqref{N0} that
$
N_0(\cut{\lambda}, \cut{\mu}, \cut{\nu}) =\frac{|\cut{\lambda}|+\cut{\lambda}_1+|\cut{\mu}|+\cut{\mu}_1+|\cut{\nu}|+\cut{\nu}_1}{2}.
$
Observe that $|\cut{\lambda}|=|\cutt{\lambda}|+ (\lambda'_1-1)$ and 
$$
\cut{\lambda}_1=
\left\lbrace
\begin{matrix}
\cutt{\lambda}_1&+1 & \text { if $\ell(\lambda) \geq 2$,}\\
\cutt{\lambda}_1 &  & \text { if $\ell(\lambda) = 1$.}
\end{matrix}
\right.
$$
Likewise for $\mu$ and $\nu$ instead of $\lambda$. The lemma follows. 
Additionally we see that the inequality is actually an equality, except when at least one of the partitions has only one row.
\end{proof}

\begin{corollary}\label{corollary hook m}
Let  $\lambda$, $\mu$ and $\nu$ be  non--empty partitions  of the same weight. The sequence of Kronecker coefficients
$
\K{\lambda\op{n}{n}}{\mu\op{n}{n}}{\nu\op{n}{n}}
$
stabilizes to  $\SSK{\cutt{\lambda}}{\cutt{\mu}}{\cutt{\nu}}$.
\end{corollary}
\begin{proof}
This corresponds to $(a,b,c,m)=(n,n,n,2n)$ and fulfills all inequalities in \eqref{translation of E} for $n \gg 0$.
\end{proof}

\begin{remark}
For $(a,b,c,m)=n\cdot(1,1,1,2)$ we have $\ell_i=n$ and $m-(a+b+c)/2=n/2$. Therefore the stable behavior in Corollary \ref{corollary hook m} takes place already for $n \geq \max(2\,d, d_1, d_2, d_3)$.
\end{remark}

\subsection{Another approach to the hook stability property, derived from Murnaghan's stability and conjugation.}\label{alternative subsection}

In this section we show that using only the well--known invariance of the
Kronecker coefficients under conjugating two of their three indexing partitions
(see for instance \cite{Macdonald, Stanley}),
\begin{equation}\label{eq:symeq}
	 \K{\lambda}{\mu}{\nu}=\K{\lambda'}{\mu'}{\nu}=\K{\lambda'}{\mu}{\nu'}=\K{\lambda}{\mu'}{\nu'},
\end{equation}
it is not difficult to prove Theorem \ref{thm:hook stab} in a special case. 

Namely, one derives from the symmetry property in \eqref{eq:symeq}, in an elementary way, that for any three
partitions $\lambda$, $\mu$, $\nu$ of the same weight,  there exists integers
$d$, $d_1$, $d_2$, $d_3$ such that \eqref{KRKSSK} holds when \eqref{translation
of E} holds with \emph{additional condition}  that $a+b+c \equiv 0 \mod 2$. 
To recover the full theorem, it would be enough to establish that there exists $m$ big enough such that
\begin{multline}\label{enough}
\K{\lambda\op{2m}{2m}}{\mu\op{2m}{2m}}{\nu\op{2m}{2m}} 
=\\
\K{\lambda\op{2m+1}{2m+1}}{\mu\op{2m+1}{2m+1}}{\nu\op{2m+1}{2m+1}} .
\end{multline}

\begin{conjecture}\label{conj mono}
For any three partitions $\lambda$, $\mu$ and $\nu$ of the same weight, and any $(a,b,c,m)$ fulfilling the inequalities
\begin{equation}\label{coneD}
\begin{matrix}
\ell_i \geq 0 & \text{ for all $i \in \{1,2,3\}$,}\\
m \geq a + b + c.
\end{matrix}
\end{equation}
there is
\begin{equation}\label{eqmono}
\K{\lambda}{\mu}{\nu} 
\leq 
\K{\lambda\op{m-a}{a}}{\mu \op{m-b}{b}}{\nu\op{m-c}{c}}.
\end{equation}
\end{conjecture}

Again, using the symmetries of the Kronecker coefficients, it is not difficult
to prove this conjecture in a restricted case, namely that \eqref{eqmono} holds for all partitions $\lambda$, $\mu$, $\nu$ of the same weight and all $(a,b,c,m)$ fulfilling \eqref{coneD} and, additionally, that $a +b +c \equiv 0 \mod 2$.

Therefore Conjecture \ref{conj mono} is equivalent to the following 
seemingly much weaker statement. 
\begin{conjecture}[Equivalent form of Conjecture \ref{conj mono}]\label{conj 111}
For any three partitions $\lambda$, $\mu$ and $\nu$ of the same weight,
\[
\K{\lambda}{\mu}{\nu} \leq \K{\lambda\op{1}{1}}{\mu\op{1}{1}}{\nu\op{1}{1}}.
\]
\end{conjecture}

\begin{remark}
Conjecture \ref{conj 111} was checked by computer, with SAGE \cite{SAGE}, for all triples of partitions of weight at most $16$.
\end{remark}

\begin{remark}
A proof of Conjecture \ref{conj 111} would provide an alternative proof of Theorem \ref{thm:hook stab}. 

Indeed, assuming Conjecture \ref{conj 111}, we have the inequalities
\begin{multline}
\K{\lambda\op{2m}{2m}}{\mu\op{2m}{2m}}{\nu\op{2m}{2m}} \\
\le \K{\lambda\op{2m+1}{2m+1}}{\mu\op{2m+1}{2m+1}}{\nu\op{2m+1}{2m+1}}\\
\le 
 \K{\lambda\op{2m+2}{2m+2}}{\mu\op{2m+2}{2m+2}}{\nu\op{2m+2}{2m+2}}.
\end{multline}
The two bounds in this inequality are equal for $m$ big enough by the hook
stability property proved using only the invariance of the Kronecker
coefficients under conjugation in \eqref{eq:symeq}.  Then \eqref{enough} would follow.
\end{remark}


\section{The second row} \label{section2row}

In this section, we describe the asymptotic behavior of some sequences of Kronecker coefficients 
$\K{\lambda + n \alpha}{\mu+n \beta}{\nu+n\gamma}$
where the integer $n$ varies, and the partitions $\alpha$, $\beta$ and $\gamma$ have at most two parts. To this end, we move to the setting of the reduced Kronecker coefficients.

We first consider in Section \ref{2row for reduced} the family of reduced Kronecker coefficients $\RK{(a,\alpha)}{(b,\beta)}{(c,\gamma)}$ where the first parts $a$, $b$ and $c$ vary arbitrarily, while the remaining parts $\alpha$, $\beta$, $\gamma$ are fixed. We obtain for these coefficients, when $a$, $b$ and $c$ are big enough, quasipolynomial formulas in $a$, $b$, $c$, of degree at most $1$ and period at most $2$. This generalizes Proposition \ref{files and columns} corresponding to $\alpha$, $\beta$, $\gamma$ equal to the empty partition. 

We determine in Section \ref{A zero} the vanishing of the generic leading coefficient $A_{\alpha,\beta,\gamma}$ in these formulas.
In Section \ref{asymptotics reduced}, we describe the asymptotic behavior of sequences of reduced Kronecker coefficients $\RK{\lambda + n\cdot (a)}{\mu + n\cdot (b)}{\nu + n\cdot (c)}$ with the partitions $\lambda$, $\mu$, $\nu$ and the integers $a$, $b$, $c$ fixed, while $n$ varies. 

The asymptotic behaviors of the corresponding sequences of Kronecker coefficients is then derived in Section \ref{asymptotics}.

\subsection{For reduced Kronecker coefficients}\label{2row for reduced}

In this section, we obtain quasipolynomial formulas in $a$, $b$, $c$ for some reduced Kronecker coefficients $\RK{(a,\alpha)}{(b,\beta)}{(c,\gamma)}$, with $\alpha$, $\beta$, $\gamma$ fixed.

\begin{theorem} \label{2row}
Let $\alpha$, $\beta$ and $\gamma$ be three partitions. There exists integers $k'_1$, $k'_2$, $k'_3$ and $A_{\alpha,\beta,\gamma}$, $B_{\alpha,\beta,\gamma}$ and $C_{\alpha,\beta,\gamma}$, such that whenever $a \geq \alpha_1$, $b \geq \beta_1$, $c \geq \gamma_1$ and 
\begin{align}\label{ineq 2rows}
\begin{matrix}
a -b&\geq& k'_1,\\
a-c &\geq& k'_2,\\
b+c-a &\geq& k'_3
\end{matrix}
\end{align}
we have 
\[
\RK{(a,\alpha)}{(b,\beta)}{(c,\gamma)}= 
\frac{1}{2} A_{\alpha,\beta,\gamma} \cdot (b+c-a)+ B_{\alpha,\beta,\gamma} +
\left\lbrace
\begin{matrix}
0 & \text{ for $b+c-a$ even,}\\
C_{\alpha,\beta,\gamma}/2  & \text{ for $b+c-a$ odd.}
\end{matrix}
\right.
\]
\end{theorem}

\begin{proof}
	For any $a$, $b$, $c$, we set from \eqref{symmetric Brion formula}
\[
\srow_{a,b,c}=\scalar{\sigma[F(X,Y,Z)]}{{s}_{(a,\alpha)}[X] {s}_{(b,\beta)}[Y] {s}_{(c,\gamma)}[Z]}
\]
where $F(X,Y,Z)=XYZ + XY + XZ + YZ$. 
When  $a$, $b$ and $c$ are at least $\alpha_1$, $\beta_1$ and $\gamma_1$ respectively, we have that  $\srow_{a,b,c}=\RK{(a,\alpha)}{(b,\beta)}{(c,\gamma)}$. 

Consider 
$\Srow_{\alpha,\beta,\gamma}(x,y,z)=\sum_{a,b,c} \srow_{a,b,c} x^a y^b z^c$.
Then, $\Srow_{\alpha,\beta,\gamma}(x,y,z)$ is equal to 
\[
\scalar{\sigma[F(X,Y,Z)]}{\vertexop{x}{X} s_{\alpha}[X] \vertexop{y}{Y} s_{\beta}[Y]\vertexop{z}{Z} s_{\gamma}[Z]
}.
\] 
This is the specialization of $\Phi_{\alpha,\beta,\gamma}$ (see Lemma \ref{main
lemma}) at $X'=x$, $Y'=y$ and $Z'=z$.

Let $\polrow_{\alpha,\beta,\gamma}(x,y,z)$ be the image of
$Q_{\alpha,\beta,\gamma}$ from Lemma \ref{main lemma} under the same
specialization.  From Lemma \ref{main lemma}, 
$
\Srow_{\alpha,\beta,\gamma}=
\sigma[xyz+xy+xz+yz] \cdot \polrow_{\alpha,\beta,\gamma}.
$
Set $\vvar^{(a,b,c)}=x^a y^b z^c$.
Write $\polrow_{\alpha,\beta,\gamma}$ as a sum of monomials,  $\polrow_{\alpha,\beta,\gamma}=\sum_{\omega \in \Omega} q_{\omega} \vvar^{\omega}$, with $\Omega$ the support of $\polrow_{\alpha,\beta,\gamma}$.
From \eqref{gseries row expanded}, 
\[
\sigma[xyz+xy+xz+yz] =
\sum_{\theta \in \mathcal{C} \cap \NN^3} \rk{\theta} \vvar^{\theta}
\]
with
$\rk{\theta}=1+\left[\min\{\ell_1(\theta),\ell_2(\theta),\ell_3(\theta)\}/2\right]$,
and it follows that
\[
\Srow_{\alpha,\beta,\gamma}(x,y,z)=\sum_{\omega \in \Omega, \theta \in \mathcal{C} \cap \NN^3} q_{\omega} \rk{\theta}
\vvar^{\omega+\theta}.
\]
For any $\tau=(a,b,c) \in \ZZ^3$, we therefore have $\srow_\tau= \sum q_{\omega} \rk{\tau-\omega}$
where the sum is over all $\omega \in \Omega$ such that $\tau-\omega \in \mathcal{C}$.
Let $\mathcal{C}_1$ be the cone defined by the inequalities 
\begin{equation}\label{coneC1}
\begin{matrix}
\ell_1 \geq 0,\\
\ell_2 \geq \ell_1,\\
\ell_3 \geq \ell_1,
\end{matrix}
\;\;\;\;\; \textrm{ or, equivalently, }\;\;\;\;\;
\begin{matrix}
b+c \geq a,\\
a \geq b,\\
a \geq c.
\end{matrix}
\end{equation}
where, as usual, the parameters $\ell_i$ are defined as in \eqref{elli}.
If $\omega$ is such that $\tau-\omega \in \mathcal{C}_1$, then we have 
\begin{align*}
\rk{\tau-\omega}
&=1+\left[\frac{\ell_1(\tau-\omega)}{2}\right]\\
&=1+\ell_1(\tau)/2-\ell_1(\omega)/2-
\left\lbrace \begin{matrix}
	0   & \text{ if } \ell_1(\omega) \equiv \ell_1(\tau) \pmod{2}\\
	1/2& \text{ if } \ell_1(\omega) \not \equiv \ell_1(\tau) \pmod{2}
\end{matrix}\right.
\end{align*}
Therefore, if $\tau$ fulfills 
$\tau-\omega \in \mathcal{C}_1$ for all $\omega$ in $\Omega$, 
then we have
\begin{align*}
\srow_\tau 
&= \sum_{\omega \in \Omega} q_{\omega} \rk{\tau-\omega}
=\sum_{\omega \in \Omega} q_{\omega} \left(1+\left[\frac{\ell_1(\tau-\omega)}{2}\right]\right)\\
&=\sum_{\omega \in \Omega} q_{\omega} 
+ \frac{1}{2}\sum_{\omega \in \Omega} q_{\omega} \ell_1(\tau) - \frac{1}{2} \sum_{\omega \in \Omega} q_{\omega} \ell_1(\omega) - \frac{1}{2} \sum_{\substack{\omega \,:\, \ell_1(\omega) \not\equiv \\\ell_1(\tau) \mod 2}} q_{\omega}.
\end{align*}
Note that $|\omega| \equiv \ell_1(\omega) \pmod{2}$ 
for all $\omega \in \ZZ^3$. The condition  $\ell_1(\omega) \not\equiv
\ell_1(\tau) \pmod{2}$ in the last sum can therefore be replaced with $|\omega|
\not\equiv |\tau| \pmod{2}$. 

Set
\[
A  =\sum_{\omega \in \Omega} q_{\omega}=\polrow(1,1,1),\quad
K  =\sum_{\omega \in \Omega} q_{\omega} \ell_1(\omega)
\]
and
\[
A^+=\sum_{\omega:|\omega| \text{ even}} q_{\omega},\quad
A^-=\sum_{\omega:|\omega| \text{ odd}} q_{\omega}. 
\]
We have obtained 
\begin{align*}
\srow_\tau
&= A + \frac{A}{2} \ell_1(\tau) - \frac{K}{2} -
\left\lbrace
\begin{array}{rl}
A^-/2 & \text{ if $\ell_1(\tau)$ is even,}\\
A^+/2 & \text{ if $\ell_1(\tau)$ is odd.}
\end{array}
\right.\\
&= A + \frac{A}{2} \ell_1(\tau) - \frac{K}{2} - \frac{A^-}{2} - 
\left\lbrace
\begin{array}{cl}
0 & \text{ if $\ell_1(\tau)$ is even,}\\
(A^+-A^-)/2 & \text{ if $\ell_1(\tau)$ is odd.}
\end{array}
\right.
\end{align*}
Set $A_{\alpha,\beta,\gamma}=A$, $B_{\alpha,\beta,\gamma}=A-K/2-A^-/2$ and
$C_{\alpha,\beta,\gamma}=A^+-A^-$. The formula in the theorem is obtained. Note
that $B_{\alpha,\beta,\gamma}$ is an integer since $K \equiv A^- \pmod{2}$. Indeed, 
\[
K=\sum_{\omega} q_{\omega} \ell_1(\omega) 
\equiv \sum_{\omega} q_{\omega} |\omega|
\equiv \sum_{\omega: |\omega| \text{ odd}} q_{\omega} \pmod{2}
\]

To conclude, observe that the condition
$\tau-\omega \in \mathcal{C}_1$ for all $\omega$ in $\Omega$, 
can be rewritten as
\[
\begin{array}{rcl}
\ell_1(\tau) &\geq& \max_{\omega \in \Omega} \ell_1(\omega),\\
\ell_1(\tau)-\ell_2(\tau) &\geq& \max_{\omega \in \Omega} \left( \ell_1(\omega)-\ell_2(\omega)\right), \\
\ell_1(\tau)-\ell_3(\tau) &\geq& \max_{\omega \in \Omega} \left( \ell_1(\omega)-\ell_3(\omega)\right),
\end{array}
\]
which is equivalent to \eqref{ineq 2rows}.  This proves the theorem.
\end{proof}

\begin{remark}\label{rem:3row}
Further computations show that one can take for the $k'_i$ in Theorem \ref{2row}
\[
\begin{array}{rcl}
k'_1 &=& |\alpha|+|\beta|+|\gamma| + \beta_1 , \\
k'_2 &=&  |\alpha|+|\beta|+|\gamma| + \gamma_1 , \\
k'_3 &=&  |\alpha|+|\beta|+|\gamma| + \alpha_1+ \beta_1 + \gamma_1.
\end{array}
\]
\end{remark}

Theorem \ref{2row again} has the following immediate corollary for Kronecker coefficients.
\begin{corollary}
Let $k'_1$, $k'_2$ and $k'_3$ be as in Theorem \ref{2row}.

For all partitions $\lambda$, $\mu$, $\nu$ of the same weight $N$, fulfilling the conditions
\[
\begin{matrix}
\lambda_2-\mu_2&\geq& k'_1,\\
\lambda_2-\nu_2 &\geq& k'_2,\\
\mu_2+\nu_2-\lambda_2 &\geq& k'_3,\\
N -\lambda_2-\mu_2-\nu_2 &\geq& (|\cut{\cut{\lambda}}|+ |\cut{\cut{\mu}}|+
|\cut{\cut{\nu}}|)/2,
\end{matrix}
\]
we have that
\[
\K{\lambda}{\mu}{\nu}= 
\frac{1}{2} A_{\cutcutlmn} \cdot (\mu_2+\nu_2-\lambda_2)+ B_{\cutcutlmn} +
\left\lbrace
\begin{matrix}
0 & \text{ for $\lambda_2+\mu_2+\nu_2$ even,}\\
C_{\cutcutlmn}/2  & \text{ for $\lambda_2+\mu_2+\nu_2$ odd.}
\end{matrix}
\right.
\]
\end{corollary}

\begin{proof}
The condition 
\[
N -\lambda_2-\mu_2-\nu_2 \geq (|\cut{\cut{\lambda}}|+ |\cut{\cut{\mu}}|+ |\cut{\cut{\nu}}|)/2
\]
ensures that $N \geq N_0(\cut{\lambda}, \cut{\mu}, \cut{\nu})$, so that
$\K{\lambda}{\mu}{\nu}=\RK{\cut{\lambda}}{\cut{\mu}}{\cut{\nu}}$, as in Section
\ref{Murnaghan}. Applying Theorem \ref{2row again} gives the result.
\end{proof}

\subsection{When is $A_{\alpha,\beta,\gamma}$ equal to zero?}\label{A zero}

From Theorem \ref{2row}, the coefficient $A_{\alpha,\beta,\gamma}$  is the generic leading term of the expression of $\RK{(a,\alpha)}{(b,\beta)}{(c,\gamma)}$ that is quasipolynomial of degree $1$ in $a$, $b$ and $c$. It is relevant to ask when it vanishes.

We will need the following lemma. 
\begin{lemma}\label{inherited}
Let $\lambda$, $\mu$, $\nu$ and $\alpha$, $\beta$, $\gamma$ be partitions.
\begin{enumerate}
\item If $\K{\alpha}{\beta}{\gamma} \neq 0$ then $\K{\lambda+\alpha}{\mu+\beta}{\nu+\gamma} \geq \K{\lambda}{\mu}{\nu}$.
\item If $\RK{\alpha}{\beta}{\gamma} \neq 0$ then $\RK{\lambda+\alpha}{\mu+\beta}{\nu+\gamma} \geq \RK{\lambda}{\mu}{\nu}$.
\end{enumerate}
\end{lemma}

\begin{proof}
For the first assertion see  \cite{Manivel:asymptotics1}.

The second assertion follows from the first one as follows. Suppose that $\RK{\alpha}{\beta}{\gamma} \neq 0$. There exist integers $a$, $b$ and $c$ such that $ \RK{\alpha}{\beta}{\gamma} = \K{(a,\alpha)}{(b,\beta)}{(c,\gamma)}$. In particular this Kronecker coefficient is non--zero.  Let $p$, $q$ and $r$ be integers such that $(p,\lambda)$, $(q,\mu)$, $(r,\nu)$ are partitions of the same weight. Then we have, for all $n \geq 0$,
\[
\K{(a+p+n,\lambda+\alpha)}{(b+q+n,\mu+\beta)}{(c+r+n,\nu+\gamma)} \geq \K{(p+n,\lambda)}{(q+n,\mu)}{(r+n,\nu)}.
\]
Taking $n$ big enough, so that both Kronecker coefficients coincide with the corresponding reduced Kronecker coefficient, we get
$\RK{\lambda+\alpha}{\mu+\beta}{\nu+\gamma} \geq \RK{\lambda}{\mu}{\nu}$.
\end{proof}

\begin{proposition}
Let $\alpha$, $\beta$, $\gamma$ be partitions. The coefficient
$A_{\alpha,\beta,\gamma}$ is zero if and only if all Kronecker coefficients
$\K{(a_1,a_2,\alpha)}{(b_1,b_2,\beta)}{(c_1,c_2,\gamma)}$ are zero for all $a_1 \geq a_2 \geq \alpha_1$, $b_1 \geq b_2 \geq \beta_1$, and $c_1 \geq c_2 \geq \gamma_1$.
\end{proposition}

\begin{proof}
It is enough to show that $A_{\alpha,\beta,\gamma}$ is zero if and only if all reduced Kronecker coefficients $\RK{(a_2,\alpha)}{(b_2,\beta)}{(c_2,\gamma)}$ are zero since, on the first hand, we have always 
$\K{(a_1,a_2,\alpha)}{(b_1,b_2,\beta)}{(c_1,c_2,\gamma)} \leq
\RK{(a_2,\alpha)}{(b_2,\beta)}{(c_2,\gamma)}$, as in see Section \ref{Murnaghan}; and, on the other hand, any reduced Kronecker coefficient $\RK{(a_2,\alpha)}{(b_2,\beta)}{(c_2,\gamma)}$ is equal to some Kronecker coefficient $\K{(a_1,a_2,\alpha)}{(b_1,b_2,\beta)}{(c_1,c_2,\gamma)}$. 

Assume that all reduced Kronecker coefficients  $\RK{(a_2,\alpha)}{(b_2,\beta)}{(c_2,\gamma)}$ are zero. 
This is the case, in particular, for the coefficients
$\RK{(n,\alpha)}{(n,\beta)}{(n,\gamma)}$. But, from Theorem \ref{2row again}, for $n$ big enough.
\[
\RK{(n,\alpha)}{(n,\beta)}{(n,\gamma)} = \frac{A_{\alpha,\beta,\gamma}}{2} n + \text{ a bounded term.} 
\]
Then $A_{\alpha,\beta,\gamma}$ must be zero.

Assume now that there exists some reduced Kronecker  coefficient   $\RK{(a,\alpha)}{(b,\beta)}{(c,\gamma)}$ that is non--zero. After Lemma \ref{inherited}, we have, for all $n \geq 0$,
$
\RK{(a+n, \alpha)}{(b+n, \beta)}{(c+n, \gamma)} \geq \RK{(n)}{(n)}{(n)}
$
On the other hand,  for $n$ big enough,
$\RK{(a+n,\alpha)}{(b+n,\beta)}{(c+n,\gamma)}=\frac{A_{\alpha,\beta,\gamma}}{2} n +
\text{ a bounded term}$. But  $\RK{(n)}{(n)}{(n)} \sim \frac{n}{2}$ from Proposition
\ref{files and columns}. Whence, necessarily $A_{\alpha,\beta,\gamma} \neq 0$.

\end{proof}


\subsection{Asymptotics of some sequences of reduced Kronecker coefficients $\RK{\lambda+n(a)}{\mu+n(b)}{\nu+n(c)}$}\label{asymptotics reduced}

We consider the asymptotic behavior of the sequence with general term $\RK{\lambda+n(a)}{\mu+n(b)}{\nu+n(c)}$, where $(\lambda, \mu, \nu)$ is a  fixed triple of partitions, and $((a),(b),(c))$ a fixed  triple of partitions with at most one part, and $n(a)$ is  $n$ times the one--part partition $(a)$.

We will set $\theta_0=(\lambda_1, \mu_1, \nu_1)$ and $\theta=(a,b,c)$, so that $\theta_0 + n \theta=(\lambda_1 + n  a, \mu_1 + n b, \nu_1 + n c)$.

Three cases will be examined, corresponding to the position of $\theta=(a,b,c)$ with respect to the cone $\mathcal{C}$: outside, on the border or in the interior. We will be able to say even more when $\theta$ is in the interior of the smaller cone $\mathcal{C}_1$ defined in \eqref{coneC1}.

We will use the following decomposition from the proof of Theorem \ref{2row}:
\[
\RK{\lambda+n(a)}{\mu+n(b)}{\nu+n (c)}=\srow_{\theta}=\sum_{\omega \in \Omega} q_{\omega} \rk{\theta_0+n \theta-\omega}
\]
where $\polrow_{\cutlmn}=\sum_{\omega \in \Omega} q_{\omega} \vvar^{\omega}$ and $\rk{\tau}=1+[\min_i \ell_i(\tau)/2)]$ if $\tau \in \mathcal{C}$, and $0$ else. The formula writes more explicitly:
\begin{equation}\label{sumP}
\RKO{\theta_0+n\theta}=\sum q_{\omega} \left(1+\left[\min_i \frac{\ell_i(\theta_0-\omega)+n \ell_i(\theta)}{2}\right]\right),
\end{equation}
where now the sum is over all $\omega$ such that: $\ell_i(\theta_0 - \omega)+n \ell_i(\theta) \geq 0$ for all $i$.

\subsubsection*{When $\theta$ is outside $\mathcal{C}$.} 

This means that $\ell_i(\theta)< 0$ for some $i$. Then $\ell_i(\theta_0-\omega)+ n \ell_i(\theta)$ is $< 0$ for $n \gg 0$. The sum \eqref{sumP} becomes empty. As a consequence, in this case, $\RKO{\theta_0+n\theta} =0$ for $n \gg 0$.

\subsubsection*{When $\theta$ is on the border of $\mathcal{C}$.}
This means that all inequalities $\ell_i(\theta) \geq 0$ are fulfilled, but at least one of them is an equality. For $i$ and $j$ such that $\ell_i(\theta) > 0$ and $\ell_j(\theta) = 0$, we have 
$\ell_i (\theta_0 -\omega)+n \ell_i(\theta) > \ell_j(\theta_0-\omega)+n \ell_j(\theta) =  \ell_j(\theta_0-\omega)$ for $n \gg 0$. Therefore all terms $\rk{\theta_0+n \theta-\omega}$ are independent on $n$. Also the sum is restricted to all $\omega$ such that $\ell_i(\theta_0-\omega) + n \ell_i(\theta) \geq 0$ for all $i$. For $i$ such that $\ell_i(\theta) > 0$, this is automatically fulfilled for $n \gg 0$; there only remains the condition $\ell_j(\theta_0 -\omega) +n \ell_j(\theta ) \geq 0$ for all $j$ such that $\ell_j(\theta) = 0$. This condition is actually independent on $n$. This shows that $\RKO{\theta_0+n\theta}$ is eventually constant in this case.

\subsubsection*{When $\theta$ is in the interior of $\mathcal{C}$.}

This means that $\ell_i(\theta)> 0$ for all $i$. 
For $n \gg 0$, the inequalities $\ell_i(\theta_0 - \omega)+n\ell_i(\theta) \geq 0$ are fulfilled for all $\omega \in \Omega$. 

We can assume, without loss of generality, that $\ell_1(\theta) \leq \ell_2(\theta)$ and $\ell_1(\theta) \leq \ell_3(\theta)$. Then, for all $\omega$,  we have for $n \gg 0$, that 
\begin{multline*}
q_{\omega} \cdot  \left(1+\left[\min_i( \ell_i(\theta_0-\omega)+n\ell_i(\theta))/2\right]\right)
\\=
q_{\omega} \ell_1(\theta)/2 \cdot n + \text{ a periodic term in $n$ with period at most $2$.}
\end{multline*}
Summing over all $\omega \in \Omega$ we get
\[
\RKO{\theta_0+n\theta} = \frac{A_{\cutlmn}}{2} \ell_1(\theta) \cdot n + \text{ a periodic term in $n$ with period at most $2$.}
\]

\subsubsection*{When $\theta$ is in the interior of $\mathcal{C}_1$.}
Then we can apply Theorem \ref{2row} with $(\lambda_1+na,\mu_1+nb,\nu_1+nc)$ instead of $(a,b,c)$.

Let us state the results obtained in a theorem.
\begin{theorem} \label{2row again}
Let $\lambda$, $\mu$ and $\nu$ be three partitions and $(a,b,c) \in \NN^3$. Without loss of generality, we can assume that $\max(a,b,c)=a$. 
Suppose that there exists $n$ such that $\RK{\lambda+n(a)}{\mu+n(b)}{\nu+n(c)}$ is non--zero. Then $A_{\cutlmn}$ is nonzero, and
\begin{enumerate}
\item if $(a,b,c)$ is in the interior of $\mathcal{C}$, then 
\[
\RK{\lambda+n(a)}{\mu+n(b)}{\nu+n(c)} \sim_{n \rightarrow \infty} \frac{A_{\cutlmn} \cdot (b+c-a)}{2} \cdot n
\]
and the difference is a periodic term in $n$ with period at most $2$.
\item if, besides,  $(a,b,c)$ is in the interior of $\mathcal{C}_1$,  then the periodic term is 
\[
\left\lbrace
\begin{matrix}
B_{\cutlmn} &-C_{\cutlmn}/2 & \text{ if $n$ and $a+b+c$ are both odd,}\\ 
B_{\cutlmn} &              & \text{ otherwise.} 
\end{matrix}
\right.
\]
\item if $(a,b,c)$ is on the border of $\mathcal{C}$ then $\RK{\lambda+n(a)}{\mu+n(b)}{\nu+n(c)}$ is eventually constant.
\item if $(a,b,c) \not\in \mathcal{C}$ then $\RK{\lambda+n(a)}{\mu+n(b)}{\nu+n(c)}=0$ for $n \gg 0$.
\end{enumerate}
\end{theorem}


\subsection{Asymptotics of some sequences of Kronecker coefficients.}\label{asymptotics}

Set $\SetL$ for the set of triples of partitions $(\lambda, \mu, \nu)$ such that
\[
|\lambda|=|\mu|=|\nu| \geq  N_0(\cutlmn).
\]
Note that the set $\SetL$ is stable under sum, and that $(\lambda, \mu, \nu) \in \SetL$ implies that $ \K{\lambda}{\mu}{\nu}=\RK{\cut{\lambda}}{\cut{\mu}}{\cut{\nu}}$.
Theorem \ref{2row again} has the following immediate consequence for Kronecker coefficients.
\begin{corollary}\label{cor:KroPlus2rows}
Let $(\lambda,\mu,\nu)$ and $(\alpha,\beta,\gamma)$ be two triples of partitions in $\SetL$, with $\alpha$, $\beta$ and $\gamma$ with at most two parts.

Without loss of generality, we assume that $\max(\alpha_2,\beta_2,\gamma_2)=\alpha_2$.

Assume that there exists $n$ such that $\K{\lambda+n \alpha}{\mu+n \beta}{\nu+n \gamma}$ is non--zero. 

Then, $A_{\cutcutlmn}$ is nonzero, and 
\begin{enumerate}
\item if $(\alpha_2,\beta_2,\gamma_2)$ is in the interior of $\mathcal{C}$  then  
\[
\K{\lambda+n \alpha}{\mu+n \beta}{\nu+n \gamma} \sim_{n \rightarrow \infty}
\frac{A_{\cutcutlmn} \cdot (\beta_2+\gamma_2-\alpha_2)}{2} \cdot n
\]
and the difference is periodic in $n$ with period at most $2$.
\item if besides, $(\alpha_2,\beta_2,\gamma_2)$ is in the interior of $\mathcal{C}_1$, then the periodic term is 
\[
\left\lbrace
\begin{matrix}
B_{\cutcutlmn}  & -C_{\cutcutlmn}/2 & \text{ if $n$ and $a+b+c$ are both odd,}\\ 
B_{\cutcutlmn}  &                  &\text{ otherwise.} 
\end{matrix}
\right.
\]
\item if $(\alpha_2,\beta_2,\gamma_2)$ is on the border of $\mathcal{C}$ then $\K{\lambda+n \alpha}{\mu+n \beta}{\nu+n \gamma}$ is eventually constant.
\item if $(\alpha_2,\beta_2,\gamma_2) \not\in \mathcal{C}$ then $\K{\lambda+n \alpha}{\mu+n \beta}{\nu+n \gamma}=0$ for $n \gg 0$.
\end{enumerate}
\end{corollary}
\begin{proof}
Since $(\lambda, \mu, \nu)$ and $(\alpha, \beta, \gamma)$ are in $\SetL$, so are all triples of partitions $(\lambda + n \alpha, \mu + n\beta, \nu + n\gamma)$ for all $n \geq 0$. Therefore,
 $\K{\lambda+n \alpha}{\mu+ n \beta}{\nu+ n \gamma} 
= 
\RK{\cut{\lambda}+n(\alpha_2)}{\cut{\mu}+n(\beta_2)}{\cut{\nu}+n(\gamma_2)}$ for all $n \geq 0$. Then we can apply Theorem \ref{2row again}.
\end{proof}

\begin{remark}We make a few remarks about Corollary \ref{cor:KroPlus2rows}.

\begin{enumerate}
\item Statement (2) is a particular case of a much more general statement: 
given any three partitions $\alpha$, $\beta$ and $\gamma$, such that $g(\alpha, \beta, \gamma) >0$, the sequence with general term $\K{\lambda+n\alpha}{\mu+n\beta}{\nu+n\gamma}$  is eventually constant, for all triples of partitions $(\alpha, \beta, \gamma)$ of the same weight,  if and only if  $\K{n\alpha}{n\beta}{n\gamma}=1$ for all $n\ge  0$.
See \cite{Stembridge,SamSnowdenGeneralizedStability}. 
\item For partitions with length at most $2$ of the same weight, $\alpha= (m-\alpha_2, \alpha_2)$, $\beta=(m-\beta_2, \beta_2)$ and $\gamma=(m-\gamma_2, \gamma_2)$, we have that $(\alpha,\beta,\gamma) \in \SetL \Leftrightarrow m \geq \alpha_2+\beta_2+\gamma_2$.
\end{enumerate}

\end{remark}



\section{Generating series}\label{generating series}

Four families of constants were defined in the previous sections: the limits $\SSK{\alpha}{\beta}{\gamma}$ under ``hook stability'' (Section \ref{hook-stability}) and the coefficients $A_{\alpha,\beta,\gamma}$, $B_{\alpha,\beta,\gamma}$ and $C_{\alpha,\beta,\gamma}$ appearing in the quasipolynomial formulas of Section \ref{section2row}.

In this section, we provide, for these families of constants, generating series akin to the generating series for the Littlewood--Richardson coefficients
\[
\sigma[XY+XZ]=\sum_{\lambda, \mu, \nu} c_{\lambda, \mu, \nu} s_{\lambda}[X] s_{\mu}[Y] s_{\nu}[Z],
\]
for the Kronecker coefficients
\[
\sigma[XYZ]=\sum_{\lambda, \mu, \nu} \K{\lambda}{\mu}{\nu} s_{\lambda}[X] s_{\mu}[Y] s_{\nu}[Z],
\]
and for the reduced Kronecker coefficients in \eqref{symmetric Brion formula}.  

\subsection{Generating series for the coefficients $\overline{\overline{g}}$.}

We give now a generating series for the limit coefficients $\SSK{\alpha}{\beta}{\gamma}$.

\begin{theorem}\label{formula-columns}
The limit $\SSK{\alpha}{\beta}{\gamma}$ in Theorem \ref{3columns} is the coefficient of $s_{\alpha}[X]s_{\beta}[Y] s_{\gamma}[Z]$ in the  expansion, in the Schur basis, of
\[
\sigma\left[XYZ+(1-\varepsilon )(XY+XZ+YZ+X+Y+Z) \right].
\]
\end{theorem}

\begin{proof}
It follows from the proof of Theorem \ref{3columns} that
$\SSK{\alpha}{\beta}{\gamma}=\sum_{\omega \in \Omega} q_{\omega}=\polcol_{\alpha,\beta,\gamma}(1,1,1)$. After the proof of Theorem \ref{3columns},  $\polcol_{\alpha,\beta,\gamma}(x,y,z)$ is the specialization of the symmetric function $Q_{\alpha,\beta,\gamma}$ at $X'=-\varepsilon x$, $Y' = -\varepsilon y$, $Z'=-\varepsilon z$. Therefore, $\SSK{\alpha}{\beta}{\gamma}$ is the specialization of  $Q_{\alpha,\beta,\gamma}$ at $X'=-\varepsilon$, $Y' = -\varepsilon$, $Z'=-\varepsilon$. By definition of $Q_{\alpha,\beta,\gamma}$ (see Lemma \ref{main lemma}), this is the coefficient of  $s_{\alpha}[X] s_{\beta}[Y] s_{\gamma}[Z]$  in the expansion of $\sigma[H(-\varepsilon,-\varepsilon,-\varepsilon)]$. Finally,
it is straightforward to compute that 
\[
H(-\varepsilon,-\varepsilon, -\varepsilon ) = XYZ+(1-\varepsilon) (XY+XZ+YZ+X+Y+Z).
\]
\end{proof}

\begin{remark}\label{rem P111}
We have used that $\polcol_{\alpha,\beta,\gamma}(1,1,1)=\SSK{\alpha}{\beta}{\gamma}$.
Interestingly, with the specialization $x=-1$, $y=-1$, $z=-1$ we get 
$\polcol_{\alpha,\beta,\gamma}(-1,-1,-1)=\K{\alpha}{\beta}{\gamma}$.
\end{remark}

\subsection{Generating series for the coefficients $A$, $B$  and $C$.}

Here we will prove the following.
\begin{theorem}\label{prop:ABC}
Let $\alpha$, $\beta$, $\gamma$ be three partitions.

Let $\chi=\sum_{n=1}^\infty p_n$, the formal sum of all power sum symmetric functions.

Let $W=XY+XZ+YZ+X+Y+Z$. 

The coefficients $A_{\alpha,\beta,\gamma}$,   $C_{\alpha,\beta,\gamma}$, and $B_{\alpha,\beta,\gamma}$ in Theorem \ref{2row} are the coefficients of $s_{\alpha}[X] s_{\beta}[Y] s_{\gamma}[Z]$ in the expansions in the Schur basis of, respectively,
\begin{multline*}
\sigma[XYZ+2W], \,\,
 \sigma[XYZ+(1+\varepsilon)W], \text{ and} \\
\sigma[XYZ+2W] \cdot \left( \frac{3}{4}+\frac{1}{4}\sigma[(\varepsilon-1) W]-\frac{1}{2} \chi[W] + \chi[YZ-X] \right)
\end{multline*} 
\end{theorem}

\begin{proof}
Let us drop in this proof the indices $(\alpha,\beta,\gamma)$ of the coefficients involved: $A$, $B$, $C$, $\polrow$ stand for $A_{\alpha,\beta,\gamma}$, $B_{\alpha,\beta,\gamma}$, $C_{\alpha,\beta,\gamma}$ and $\polrow_{\alpha,\beta,\gamma}$.

With the notations of the proof of Theorem \ref{2row}, we have 
$A=\polrow(1,1,1)$, 
$B=A-K/2-A^-/2$ and 
$C=A^+-A^-$.

Remember that  $\polrow$ is the coefficient of $s_{\alpha}[X] s_{\beta}[Y] s_{\gamma}[Z]$ in the expansion in the Schur basis of $\sigma[H(x,y,z)]$. We have
\begin{multline*}
H(x,y,z)=XYZ+(1+z)XY+(1+y)XZ+(1+x)YZ\\
+(yz+y+z-1/x)X
+(xz+x+z-1/y)Y
+(xy+x+y-1/z)Z.
\end{multline*}
Specializing $x$,$y$,$z$ at $1$ we get that $A$ is the coefficient of $s_{\alpha}[X] s_{\beta}[Y] s_{\gamma}[Z]$ in the expansion in the Schur basis of  $\sigma[XYZ+2W]$.

Let us get now a generating series for the coefficients $C$. We have $C=A^+-A^-$, where $A^+$ (resp. $A^-$) is
 the sum of all coefficients $q_{\omega}$ of $\polrow$ such that $\ell_1(\omega)$ is even (resp. odd).
Note that for any $\omega \in \ZZ^3$, we have $\ell_1(\omega) \equiv |\omega| \mod 2$. Therefore, $A^+$ (resp. $A^-$) is also the sum of all coefficients $q_{\omega}$ of $\polrow$ such that $|\omega|$ is even (resp. odd). Thus
\[
\polrow(-1,-1,-1)=\sum_{\omega} q_{\omega} \; (-1)^{|\omega|}=A^+-A^-=C.
\]
Specializing the variables $x$, $y$ and $z$ at $-1$ in $\sigma[H(x,y,z)]$ (this corresponds to specializations at $\varepsilon$ as alphabets), we get that $C$ is the coefficient of $s_{\alpha}[X] s_{\beta}[Y] s_{\gamma}[Z]$ in the expansion in the Schur basis of  $\sigma[XYZ+(1+\varepsilon)W]$.

Now $B=A-K/2-A^-/2$. Since $C=A^+-A^-$ and $A=A^++A^-$, we have also $B=3 A/4+C/4-K/2$. The generating series for the coefficients $A$ and $C$ have just been obtained. Let us focus on the generating series for the coefficients $K$.

Note that
\[
K=\sum_{(a,b,c) \in \Omega} q_{a,b,c} (b+c-a)
\]
This can be obtained as $\frac{\partial \polrow(1/t,t,t)}{\partial t}_{|t=1}$.  Therefore $K$ is the coefficient of $s_{\alpha}[X] s_{\beta}[Y] s_{\gamma}[Z]$ in the expansion in the Schur basis of 
\[
\frac{\partial \sigma\left[
H(1/t,t,t)
\right]}{\partial t}_{|t=1}
\]
We compute that 
\[
H(1/t,t,t)=t^2 X + t(W-YZ)+(XYZ+W-X)+1/t YZ.
\]
We now use that $\sigma=\exp\left(\sum_{n=1}^{\infty} p_n/n\right)$, where the $p_n$ are the power sum symmetric functions (see \cite[]{Macdonald}).
We get that $\sigma[H(1/t,t,t)]$ is equal to 
\[
\exp\left(\sum_{n=1}^{\infty} \frac{p_n[X] t^{2n}+p_n[W-YZ] t^{n}+p_n[XYZ+W-X] + p_n[YZ] t^{-n}}{n}
\right).
\]
Derivating with respect to $t$ and specializing $t$ at $1$, we obtain
\[
\frac{\partial \sigma[H(1/t,t,t)]}{\partial t}_{|t=1}
=
\sigma[H(1,1,1)] \cdot \sum_{n=1}^{\infty} \left(2 p_n[X] + p_n[W-YZ]  - p_n[YZ] \right)
\]
Let $\chi=\sum_{n=1}^{\infty} p_n$. Since $H(1,1,1)=XYZ + 2 W$, we get 
\[
\frac{\partial \sigma[H(1/t,t,t)]}{\partial t}_{|t=1}
=
\sigma[XYZ+2 W] \cdot \left(\chi[W]+2 \chi[X-YZ]\right).
\]
The generating series for the coefficients $B_{\alpha, \beta, \gamma}$ is thus 
\[
\frac{3}{4}\sigma[XYZ+2W] 
+ \frac{1}{4} \sigma[XYZ+(1+\varepsilon) W]
- \frac{1}{2} \sigma[XYZ+2 W] \cdot \left(\chi[W]+2 \chi[X-YZ]\right)
\]
which is equal to
\[
\sigma[XYZ+2W] 
\left(
\frac{3}{4}+\frac{1}{4}\sigma[(\varepsilon-1) W]-\frac{1}{2} \chi[W] + \chi[YZ-X]
\right).
\]
\end{proof}

\begin{remark}
It is possible to rewrite the formula for the generating function of the coefficients $B_{\alpha,\beta,\gamma}$ in such a way that it is clearly a combination of Schur functions with integer coefficients (as we know it is, after Theorem \ref{2row}). There are many ways of doing this. One of them is:
\[
\sigma[XYZ+2W]  \cdot \left(1 - \sum_{a \text{ even}, b} (-1)^b s_{(a|b)}[W] + \sum_{a, b} (-1)^b s_{(a|b)}[YZ-X] \right)
\]
where $(a|b)$ is the partition $(1+a, 1^b)$ ("Frobenius notation" for partitions, see \cite[I. \S 1]{Macdonald})
\end{remark}

\begin{example}
One can derive from Theorem \ref{prop:ABC} the following formulas for the coefficients in the paper, when two of the three indices are the empty partition. 
\begin{align*}
&A_{(\alpha_1,\alpha_2),\ep,\ep} =\alpha_1-\alpha_2+1,\\
&C_{(\alpha_1,\alpha_2),\ep,\ep} =
\left\lbrace
\begin{matrix}
 (-1)^{\alpha_2}  & \text{ if } \alpha_1 \equiv \alpha_2\mod 2, \\
0    & \text{ otherwise.}
\end{matrix}
\right.\\
&B_{(\alpha_1,\alpha_2),\ep,\ep} \text{ is the nearest integer from } -3 \cdot \frac{ (\alpha_1)^2-(\alpha_2-1)^2}{4}
\end{align*}
and, when $(\alpha_1,\alpha_2)$ is not the empty partition,
\[
B_{\ep, (\alpha_1,\alpha_2),\ep}  \text{ is the nearest integer from } -3 \cdot \frac{ (\alpha_1-1)^2-(\alpha_2-2)^2}{4}.
\]

Similarly, one derives from Theorem \ref{formula-columns} that
\[
\SSK{\alpha}{\ep}{\ep}=
\left\lbrace
\begin{array}{cl}
2 & \text{ if $\alpha$ is a hook},\\
1 & \text{ if } \alpha=\ep,\\
0 & \text{ otherwise.}
\end{array}
\right.
\] 
\end{example}


\section{Final remarks}\label{final-remarks}

The rate of growth  experienced by the reduced Kronecker coefficients  as we add
cells to remaining row is  harder to understand. The stretched Kronecker
coefficients are  known to be described by a quasi-polynomial;  see
\cite{Manivel:asymptotics1, GCT6, BaldoniVergne}.  In particular, the following
corollary holds.  

\begin{corollary}[Manivel, \cite{Manivel:asymptotics1}] 
For any triple
$\lambda, \mu, $ and $\nu$, the stretched Kronecker coefficient
$g(k\lambda, k\mu, k\nu)$ is a quasi-polynomial function of $k\ge 0$.
\end{corollary}
Examples of these quasi-polynomial functions have been computed in \cite{ BOR-FPSAC, BaldoniVergne}.
All specializations of the form  $\lambda:=\lambda+k\mu$ will then be described by quasi-polynomials in $k$. 

Some particular instances of this problem have been studied in the literature. Recall that in \cite{Murnaghan:1938} Murnaghan observed that the reduced Kronecker coefficients  $\overline{g}^{\lambda}_{\mu,\nu}$ such that $|\lambda|= |\mu|+|\nu|$ coincide with the Littlewood-Richardson coefficients $c^{\lambda}_{\mu,\nu}$. 
For the stretched Littlewood-Richardson coefficients, it has been shown in \cite{Rassart, Derksen:Weyman} that $\overline{g}(k\lambda, k\mu, k\nu)$  is described by a polynomial (and not just a quasi-polynomial).
Moreover, the degree of the stretched Littlewood-Richardson polynomials have been studied in \cite{King:Tollu:Toumazet, King:Tollu:Toumazet:Hive}.

Other families (that are not Littlewood-Richardson coefficients) have appeared in the literature.
 For example, from the calculations appearing in  \cite{Colmenarejo-Rosas, Colmenarejo} we know that the sequence $\bar{g}^{(k)}_{(k^{a}),(k^{a})}$ is described by a  quasipolynomial of degree $2a-1$. However, the sequence $\bar{g}^{(k)}_{(2k-j,k^{a-1}),(k^a)}$, with $k \ge 2j$  is described by a quasipolynomial of degree $3a-2$. Note that the period of both quasipolynomials  divides $\ell$,  the least common multiple of  $1, 2, \ldots, a, a+1$. In fact, it has been checked that the period is exactly $\ell$ for $a \le 10$ for the first family, and for $a \le 7$ for the second one.

Note that for $a=1$, both sequences are described by a linear quasipolynomial of period 2, as predicted by our work. For $a=2$ the first sequence is described by a quasipolynomial of degree 3, whereas the second sequence is described by a quasipolynomial of degree $4$. The two resulting  quasi polynomials are copied here. 

\begin{example}
The coefficients $\overline{g}^{(k)}_{(k^2), (k^2)}$ are given by the following quasipolynomial of degree $3$ and period $6$:
\begin{align*}
\overline{g}^{(k)}_{(k^2), (k^2)} =
 \left\{
{
 \begin{array}{ll}
1/72 \, (k+6) \, (k^2+6 k+12)                           & \text{ if } k \equiv 0 \mod 6\\
 1/72 \,(k+5) \, (k^2+7 k+4)             & \text{ if } k \equiv 1 \mod 6\\
1/72 \,(k+4)^3              & \text{ if } k \equiv 2 \mod 6\\
1/72 \, (k+3) \,(k^2+9 k+12)     &\text{ if } k \equiv 3 \mod 6\\
1/72 \, (k+2) \, (k^2+10 k+28)                        &\text{ if } k \equiv 4 \mod 6\\
1/72 \, (k+1) \,(k+4)\, (k+7)                  &\text{ if } k \equiv 5 \mod 6
 \end{array}
  }
 \right.
 \end{align*}
  \end{example} 

The factorizations obtained for this families resemble those observed and studied for the stretched Littlewood-Richardson coefficients in \cite{King:Tollu:Toumazet, King:Tollu:Toumazet:Hive}.

 \begin{example}
The coefficients $\bar{g}^{(k)}_{(2k-j,k),(k^2)}$, with $k \ge 2j$, are given by the following quasipolynomial of degree $4$ and period $6$:
\begin{align*}
\bar{g}^{(k)}_{(2k-j,k),(k^2)}=
 \left\{
{ 
 \begin{array}{ll}
1/288 \, (j+6)\, (j^3+12 j^2+40 j+48)              & \text{ if } j \equiv 0 \mod 6\\
1/288  \,(j+5)^2\, (j+1)\, (j+7)           & \text{ if } j \equiv 1 \mod 6\\
  1/288 \, (j+4)^2\, (j+2) \,(j+8)          & \text{ if } j \equiv 2 \mod 6\\
 1/288 \,(j+3)\, (j^3+15 j^2+67 j+69)               &\text{ if } j \equiv 3 \mod 6\\
   1/288 \,(j+2) \,(j+4)^2\, (j+8)                       &\text{ if } j \equiv 4 \mod 6\\
   1/288 \,(j+1) \,(j+5)^2 \,(j+7)                 &\text{ if } j \equiv 5 \mod 6
 \end{array}}
 \right.
 \end{align*}
 \end{example}

 Interestingly, for both families, the sequences obtained  as the result of
 incrementing the  parameter $a$ (the sizes of the columns) is weakly increasing, and bounded.  This follows easily from the combinatorial interpretations in terms of plane partitions provided in  \cite{Colmenarejo-Rosas, Colmenarejo}.

\section*{Acknowledgments}
We thank Laura Colmenarejo and Rosa Orellana  for some interesting  comments.
The second and third authors  want to thank Birkbeck College and  the Imperial College Library for hosting us during our visit to London where most of this research was conducted.

\bibliographystyle{hplain}
\bibliography{kroproduct}

\def\cprime{$'$}
\begin{thebibliography}{10}

\bibitem{BaldoniVergne}
Velleda {Baldoni} and Mich{\`e}le {Vergne}.
\newblock {Multiplicity of compact group representations and applications to
  {K}ronecker coefficients}.
\newblock {\em ArXiv e-prints}, June 2015, 1506.02472.

\bibitem{Ballantine:Orellana}
Cristina~M. Ballantine and Rosa~C. Orellana.
\newblock A combinatorial interpretation for the coefficients in the
  {K}ronecker product {$s\sb {(n-p,p)}\ast s\sb \lambda$}.
\newblock {\em S\'em. Lothar. Combin.}, 54A:Art. B54Af, 29 pp. (electronic),
  2005/07.

\bibitem{Blasiak:hook}
Jonah {Blasiak}.
\newblock {Kronecker coefficients for one hook shape}.
\newblock {\em ArXiv e-prints}, September 2012, 1209.2018.

\bibitem{BDO}
C.~Bowman, M.~De~Visscher, and R.~Orellana.
\newblock The partition algebra and the {K}ronecker coefficients.
\newblock {\em Trans. Amer. Math. Soc.}, 367(5):3647--3667, 2015.

\bibitem{BMOR}
Emmanuel {Briand}, Peter~R.~W. {McNamara}, Rosa {Orellana}, and Mercedes
  {Rosas}.
\newblock {Commutation and normal ordering for operators on symmetric
  functions}.
\newblock {\em ArXiv e-prints}, September 2015, 1509.02581.

\bibitem{BOR-FPSAC}
Emmanuel Briand, Rosa Orellana, and Mercedes Rosas.
\newblock Quasipolynomial formulas for the {K}ronecker coefficients indexed by
  two two-row shapes (extended abstract).
\newblock In {\em {FPSAC} 2009}, Discrete Math. Theor. Comput. Sci. Proc., AK,
  pages 241--252. Assoc. Discrete Math. Theor. Comput. Sci., Nancy, 2009.

\bibitem{BOR-CC}
Emmanuel Briand, Rosa Orellana, and Mercedes Rosas.
\newblock Reduced {K}ronecker coefficients and counter-examples to {M}ulmuley's
  strong saturation conjecture {SH}.
\newblock {\em Comput. Complexity}, 18(4):577--600, 2009.
\newblock With an appendix by Ketan Mulmuley.

\bibitem{BOR-JA}
Emmanuel Briand, Rosa Orellana, and Mercedes Rosas.
\newblock The stability of the {K}ronecker product of {S}chur functions.
\newblock {\em J. Algebra}, 331:11--27, 2011.

\bibitem{Brion}
Michel Brion.
\newblock Stable properties of plethysm: on two conjectures of {F}oulkes.
\newblock {\em Manuscripta Math.}, 80(4):347--371, 1993.

\bibitem{Burgisser:Ikenmeyer}
Peter B{\"u}rgisser and Christian Ikenmeyer.
\newblock The complexity of computing {K}ronecker coefficients.
\newblock In {\em Proceedings of FPSAC 2008 (Formal Power Series and Algebraic
  Combinatorics)}, Valparaiso, 2008.

\bibitem{ThibonCarre}
Christophe Carr{\'e} and Jean-Yves Thibon.
\newblock Plethysm and vertex operators.
\newblock {\em Adv. in Appl. Math.}, 13(4):390--403, 1992.

\bibitem{Christandl:Harrow:Mitchison}
Matthias Christandl, Aram~W. Harrow, and Graeme Mitchison.
\newblock Nonzero {K}ronecker coefficients and what they tell us about spectra.
\newblock {\em Comm. Math. Phys.}, 270(3):575--585, 2007.

\bibitem{Christandl:Mitchison}
Matthias Christandl and Graeme Mitchison.
\newblock The spectra of quantum states and the {K}ronecker coefficients of the
  symmetric group.
\newblock {\em Comm. Math. Phys.}, 261(3):789--797, 2006.

\bibitem{ChurchEllenbergFarb:FI}
Thomas Church, Jordan~S. Ellenberg, and Benson Farb.
\newblock F{I}-modules and stability for representations of symmetric groups.
\newblock {\em Duke Math. J.}, 164(9):1833--1910, 2015.

\bibitem{Church:Farb:stability}
Thomas Church and Benson Farb.
\newblock Representation theory and homological stability.
\newblock {\em Adv. Math.}, 245:250--314, 2013.

\bibitem{ClausenMeier}
M.~Clausen and H.~Meier.
\newblock Extreme irreduzible {K}onstituenten in {T}ensordarstellungen
  symmetrischer {G}ruppen.
\newblock {\em Bayreuth. Math. Schr.}, 45:1--17, 1993.

\bibitem{Colmenarejo}
L.~{Colmenarejo}.
\newblock {Combinatorics on several families of Kronecker coefficients related
  to plane partitions}.
\newblock {\em ArXiv e-prints}, April 2016, 1604.00803.

\bibitem{Colmenarejo-Rosas}
Laura Colmenarejo and Mercedes Rosas.
\newblock Combinatorics on a family of reduced {K}ronecker coefficients.
\newblock {\em Comptes Rendus Mathematiques}, 353:865--869, 2015.

\bibitem{Derksen:Weyman}
Jerzy Derksen, Harm;~Weyman.
\newblock On the {L}ittlewood-{R}ichardson polynomials.
\newblock {\em J. Algebra}, 255(no. 2):247--257, 2002.

\bibitem{Dvir}
Yoav Dvir.
\newblock On the {K}ronecker product of {$S\sb n$} characters.
\newblock {\em J. Algebra}, 154(1):125--140, 1993.

\bibitem{Ikenmeyer:Mulmuley:Walter:Vanishing}
C.~{Ikenmeyer}, K.~D. {Mulmuley}, and M.~{Walter}.
\newblock {On vanishing of Kronecker coefficients}.
\newblock {\em ArXiv e-prints}, July 2015, 1507.02955.

\bibitem{Liu}
Ricky {Ini Liu}.
\newblock {A simplified Kronecker rule for one hook shape}.
\newblock {\em ArXiv e-prints}, December 2014, 1412.2180.

\bibitem{JarvisYungVertex2}
P.~D. Jarvis and C.~M. Yung.
\newblock Vertex operators and composite supersymmetric {$S$}-functions.
\newblock {\em J. Phys. A}, 26(8):1881--1900, 1993.

\bibitem{Jing:Spin}
Nai~Huan Jing.
\newblock Vertex operators, symmetric functions, and the spin group
  {$\Gamma_n$}.
\newblock {\em J. Algebra}, 138(2):340--398, 1991.

\bibitem{JingNie}
Naihuan Jing and Benzhi Nie.
\newblock Vertex operators, {W}eyl determinant formulae and {L}ittlewood
  duality.
\newblock {\em Ann. Comb.}, 19(3):427--442, 2015.

\bibitem{King:Tollu:Toumazet}
R.C. King, C.~Tollu, and F.~Toumazet.
\newblock {\em Stretched Littlewood-Richardson and Kostka coefficients.},
  volume~34 of {\em Symmetry in physics, 99--112}.
\newblock Amer. Math. Soc., Providence, RI, 2004.

\bibitem{King:Tollu:Toumazet:Hive}
{King, R. C. and Tollu, C. and Toumazet, F.}
\newblock {The hive model and the factorisation of {K}ostka coefficients},.
\newblock {\em {S\'em. Lothar. Combin.},}, {54A},:{Art. B54Ah, 22},,
  {2005/07},.

\bibitem{Kirillov:saturation}
Anatol~N. Kirillov.
\newblock An invitation to the generalized saturation conjecture.
\newblock {\em Publ. Res. Inst. Math. Sci.}, 40(4):1147--1239, 2004.

\bibitem{Klemm}
Michael Klemm.
\newblock Tensorprodukte von {C}harakteren der symmetrischen {G}ruppe.
\newblock {\em Arch. Math. (Basel)}, 28(5):455--459, 1977.

\bibitem{Klyachko}
A.~{Klyachko}.
\newblock {Quantum marginal problem and representations of the symmetric
  group}.
\newblock {\em ArXiv e-prints}, September 2004, quant-ph/0409113.

\bibitem{Klyachko:saturation}
Alexander~A. Klyachko.
\newblock Stable bundles, representation theory and {H}ermitian operators.
\newblock {\em Selecta Math. (N.S.)}, 4(3):419--445, 1998.

\bibitem{ThibonLascoux:stability}
A.~Lascoux and J.-Y. Thibon.
\newblock Vertex operators and the class algebras of symmetric groups.
\newblock {\em Zap. Nauchn. Sem. S.-Peterburg. Otdel. Mat. Inst. Steklov.
  (POMI)}, 283:156--177, 261, 2001.

\bibitem{Lascoux}
Alain Lascoux.
\newblock {\em Symmetric functions and combinatorial operators on polynomials},
  volume~99 of {\em CBMS Regional Conference Series in Mathematics}.
\newblock Published for the Conference Board of the Mathematical Sciences,
  Washington, DC; by the American Mathematical Society, Providence, RI, 2003.

\bibitem{Littlewood:1958}
D.~E. Littlewood.
\newblock Products and plethysms of characters with orthogonal, symplectic and
  symmetric groups.
\newblock {\em Canad. J. Math.}, 10:17--32, 1958.

\bibitem{Macdonald}
I.~G. Macdonald.
\newblock {\em Symmetric functions and {H}all polynomials}.
\newblock Oxford Mathematical Monographs. The Clarendon Press, Oxford
  University Press, New York, second edition, 1995.
\newblock With contributions by A. Zelevinsky, Oxford Science Publications.

\bibitem{Manivel}
Laurent Manivel.
\newblock Applications de {G}auss et pl{\'e}thysme.
\newblock {\em Ann. Inst. Fourier (Grenoble)}, 47(3):715--773, 1997.

\bibitem{Manivel:asymptotics1}
Laurent Manivel.
\newblock On the asymptotics of {K}ronecker coefficients.
\newblock {\em J. Algebraic Combin.}, 42(4):999--1025, 2015.

\bibitem{Manivel:asymptotics2}
Laurent Manivel.
\newblock On the asymptotics of {K}ronecker coefficients, 2.
\newblock {\em S\'em. Lothar. Combin.}, 75:Art. B75d, 13, 2015.

\bibitem{GCT6}
Ketan~D. Mulmuley.
\newblock Geometric complexity theory {VI}: the flip via saturated and positive
  integer programming in representation theory and algebraic geometry.
\newblock Technical Report TR--2007--04, Computer Science Department, The
  University of Chicago, May 2007.
\newblock also available as {arXiv:0704.0229}.

\bibitem{Murnaghan:1938}
Francis~D. Murnaghan.
\newblock The {A}nalysis of the {K}ronecker {P}roduct of {I}rreducible
  {R}epresentations of the {S}ymmetric {G}roup.
\newblock {\em Amer. J. Math.}, 60(3):761--784, 1938.

\bibitem{PakPanova:bounds}
Igor {Pak} and Greta {Panova}.
\newblock {Bounds on the Kronecker coefficients}.
\newblock {\em ArXiv e-prints}, June 2014, 1406.2988.

\bibitem{PakPanova:complexity}
Igor Pak and Greta Panova.
\newblock On the complexity of computing {K}ronecker coefficients.
\newblock {\em Computational Complexity}, pages 1--36, 2015.

\bibitem{Rassart}
Etienne Rassart.
\newblock A polynomiality property for {L}ittlewood-{R}ichardson coefficients.
\newblock {\em J. Combin. Theory Ser. A}, 107(2):161--179, 2004.

\bibitem{Ressayre}
N.~Ressayre.
\newblock Geometric invariant theory and the generalized eigenvalue problem.
\newblock {\em Invent. Math.}, 180(2):389--441, 2010.

\bibitem{Rosas}
Mercedes~H. Rosas.
\newblock The {K}ronecker product of {S}chur functions indexed by two-row
  shapes or hook shapes.
\newblock {\em J. Algebraic Combin.}, 14(2):153--173, 2001.

\bibitem{SamSnowdenGeneralizedStability}
Steven~V Sam and Andrew Snowden.
\newblock Stability patterns in representation theory.
\newblock {\em Forum Math. Sigma}, 3:e11, 108, 2015.

\bibitem{SamSnowden:Stembridge}
StevenV Sam and Andrew Snowden.
\newblock Proof of {S}tembridge's conjecture on stability of {K}ronecker
  coefficients.
\newblock {\em Journal of Algebraic Combinatorics}, pages 1--10, 2015.

\bibitem{ThibonScharf:innerPlethysm}
Thomas Scharf and Jean-Yves Thibon.
\newblock A {H}opf-algebra approach to inner plethysm.
\newblock {\em Adv. Math.}, 104(1):30--58, 1994.

\bibitem{ThibonScharfWybourne}
Thomas Scharf, Jean-Yves Thibon, and Brian~G. Wybourne.
\newblock Reduced notation, inner plethysms and the symmetric group.
\newblock {\em J. Phys. A}, 26(24):7461--7478, 1993.

\bibitem{Stanley}
Richard~P. Stanley.
\newblock {\em Enumerative combinatorics. {V}ol. 2}, volume~62 of {\em
  Cambridge Studies in Advanced Mathematics}.
\newblock Cambridge University Press, Cambridge, 1999.
\newblock With a foreword by Gian-Carlo Rota and appendix 1 by Sergey Fomin.

\bibitem{SAGE}
W.\thinspace{}A. Stein et~al.
\newblock {\em {S}age {M}athematics {S}oftware ({V}ersion 6.4.1)}.
\newblock The Sage Development Team, 2014.
\newblock {\tt http://www.sagemath.org}.

\bibitem{Stembridge}
John Stembridge.
\newblock Generalized stability of {K}ronecker coefficients.
\newblock Available at \url{http://www.math.lsa.umich.edu/~jrs/papers.html},
  August 2014.

\bibitem{Thibon}
Jean-Yves Thibon.
\newblock Hopf algebras of symmetric functions and tensor products of symmetric
  group representations.
\newblock {\em Internat. J. Algebra Comput.}, 1(2):207--221, 1991.

\bibitem{Vallejo:tomography}
E.~{Vallejo}.
\newblock {Stability of Kronecker coefficients via discrete tomography}.
\newblock {\em ArXiv e-prints}, August 2014, 1408.6219.

\bibitem{Vallejo}
Ernesto Vallejo.
\newblock Stability of {K}ronecker products of irreducible characters of the
  symmetric group.
\newblock {\em Electronic journal of combinatorics}, 6(1):1--7, 1999.

\bibitem{Zelevinsky:LNM}
Andrey~V. Zelevinsky.
\newblock {\em Representations of finite classical groups}, volume 869 of {\em
  Lecture Notes in Mathematics}.
\newblock Springer-Verlag, Berlin-New York, 1981.

\end{thebibliography}

\newpage
\appendix

\section{Table of coefficients}\label{apptab}

Tables \ref{part1} and \ref{part2} display the coefficients $\SSK{\alpha}{\beta}{\gamma}$, $A_{\alpha,\beta,\gamma}$, $ B_{\alpha,\beta,\gamma}$ and $C_{\alpha,\beta,\gamma}$ for all partitions $\alpha$, $\beta$ and $\gamma$ with weight at most $3$. Note that $\SSK{\alpha}{\beta}{\gamma}$, $A_{\alpha,\beta,\gamma}$ and $C_{\alpha,\beta,\gamma}$ are invariant under permutation of their three indices. This is why the table gives their values only for $\alpha \geq \beta \geq \gamma$, where the order $\geq$ is the degree lexicographic ordering. The coefficients $B_{\alpha,\beta,\gamma}$ is only invariant under permuting its last two indices. 

These coefficients where calculated by series expansion of the generating series and using SAGE \cite{SAGE}.

\begin{table}
\[
\begin{array}{ccccccccc}
\alpha & \beta & \gamma & \SSK{\alpha}{\beta}{\gamma} & A_{\alpha,\beta,\gamma} & B_{\alpha,\beta,\gamma} & B_{\beta,\alpha,\gamma} & B_{\gamma,\alpha,\beta} & C_{\alpha,\beta,\gamma} \\\hline
\ep       &\ep       &\ep       &1    &1    &1   
&1    &1    &1     \\
(1)      &\ep       &\ep       &2    &2    &   
0&1    &1    &    0 \\
(1)      &(1)      &\ep       &6    &6    &    0&
0&3    &    0 \\
(1)      &(1)      &(1)      &21   &21   &    0&
0&    0&1     \\
(2)      &\ep       &\ep       &2    &3    &-2  
&1    &1    &1     \\
(2)      &(1)      &\ep       &8    &10   &-7  
&-2   &3    &    0 \\
(2)      &(1)      &(1)      &34   &40   &-25 
&-5   &-5   &    0 \\
(2)      &(2)      &\ep       &14   &20   &-14 
&-14  &6    &2     \\
(2)      &(2)      &(1)      &66   &86   &-57 
&-57  &-14  &    0 \\
(2)      &(2)      &(2)      &145  &203  &-133
&-133 &-133 &5     \\
(2)      &(2)      &(1, 1)   &144  &150  &-84 
&-84  &-84  &-4    \\
(2)      &(2)      &(1, 1, 1)&204  &134  &-54 
&-54  &-121 &    0 \\
(2)      &(1, 1)   &\ep       &14   &12   &-8  
&-8   &4    &-2    \\
(2)      &(1, 1)   &(1)      &66   &62   &-33 
&-33  &-2   &    0 \\
(2)      &(1, 1)   &(1, 1)   &145  &131  &-55 
&-55  &-55  &5     \\
(2)      &(1, 1, 1)&\ep       &16   &6    &-3  
&-6   &3    &    0 \\
(2)      &(1, 1, 1)&(1)      &84   &46   &-19 
&-42  &4    &    0 \\
(2)      &(1, 1, 1)&(1, 1)   &206  &144  &-45 
&-117 &-45  &    0 \\
(2)      &(1, 1, 1)&(1, 1, 1)&326  &240  &-48 
&-168 &-168 &    0 \\
(1, 1)   &\ep       &\ep       &2    &1    &-1   &
0&    0&-1    \\
(1, 1)   &(1)      &\ep       &8    &6    &-3   &
0&3    &    0 \\
(1, 1)   &(1)      &(1)      &34   &28   &-13 
&1    &1    &    0 \\
(1, 1)   &(1, 1)   &\ep       &14   &12   &-4  
&-4   &8    &2     \\
(1, 1)   &(1, 1)   &(1)      &66   &54   &-21 
&-21  &6    &    0 \\
(1, 1)   &(1, 1)   &(1, 1)   &144  &110  &-38 
&-38  &-38  &-4    \\
(3)      &\ep       &\ep       &2    &4    &-6   &
0&    0&    0 \\
(3)      &(1)      &\ep       &8    &14   &-20 
&-6   &1    &    0 \\
(3)      &(1)      &(1)      &38   &59   &-78 
&-19  &-19  &1     \\
(3)      &(2)      &\ep       &16   &30   &-42 
&-27  &3    &    0 \\
(3)      &(2)      &(1)      &84   &138  &-178
&-109 &-40  &    0 \\
(3)      &(2)      &(2)      &206  &348  &-435
&-261 &-261 &    0 \\
(3)      &(2)      &(1, 1)   &204  &258  &-299
&-170 &-170 &    0 \\
(3)      &(2)      &(1, 1, 1)&320  &250  &-250
&-125 &-250 &    0 \\
(3)      &(1, 1)   &\ep       &16   &18   &-24 
&-15  &3    &    0 \\
(3)      &(1, 1)   &(1)      &84   &98   &-118
&-69  &-20  &    0 \\
(3)      &(1, 1)   &(1, 1)   &206  &220  &-235
&-125 &-125 &    0 \\
(3)      &(3)      &\ep       &22   &50   &-72 
&-72  &3    &    0 \\
(3)      &(3)      &(1)      &122  &240  &-321
&-321 &-81  &2     \\
(3)      &(3)      &(2)      &326  &640  &-820
&-820 &-500 &    0 \\
(3)      &(3)      &(1, 1)   &320  &478  &-574
&-574 &-335 &    0 \\
(3)      &(3)      &(3)      &565  &1243
&-1597&-1597&-1597&5     \\
(3)      &(3)      &(2, 1)   &1056 &1632
&-1888&-1888&-1888&    0 \\
(3)      &(3)      &(1, 1, 1)&544  &506  &-521
&-521 &-521 &-4   
\end{array}
\]
\caption{Table of the coefficients of the paper, for three indexing partitions with weight at most $3$ (part 1 of 2). }\label{part1}
\end{table}

\begin{table}
\[
\begin{array}{ccccccccc}
\alpha & \beta & \gamma & \SSK{\alpha}{\beta}{\gamma} & A_{\alpha,\beta,\gamma} & B_{\alpha,\beta,\gamma} & B_{\beta,\alpha,\gamma} & B_{\gamma,\alpha,\beta} & C_{\alpha,\beta,\gamma} \\\hline
(3)      &(2, 1)   &\ep       &38   &50   &-66 
&-66  &9    &    0 \\
(3)      &(2, 1)   &(1)      &224  &288  &-344
&-344 &-56  &    0 \\
(3)      &(2, 1)   &(2)      &610  &824  &-938
&-938 &-526 &    0 \\
(3)      &(2, 1)   &(1, 1)   &610  &700  &-738
&-738 &-388 &    0 \\
(3)      &(2, 1)   &(2, 1)   &2037 &2465
&-2515&-2515&-2515&1     \\
(3)      &(2, 1)   &(1, 1, 1)&1056 &928  &-832
&-832 &-832 &    0 \\
(3)      &(1, 1, 1)&\ep       &22   &10   &-12 
&-12  &3    &    0 \\
(3)      &(1, 1, 1)&(1)      &122  &80   &-85 
&-85  &-5   &-2    \\
(3)      &(1, 1, 1)&(1, 1)   &326  &260  &-240
&-240 &-110 &    0 \\
(3)      &(1, 1, 1)&(1, 1, 1)&565  &451  &-355
&-355 &-355 &5     \\
(2, 1)   &\ep       &\ep       &2    &2    &-3   &
0&    0&    0 \\
(2, 1)   &(1)      &\ep       &12   &12   &-15 
&-3   &3    &    0 \\
(2, 1)   &(1)      &(1)      &64   &64   &-72 
&-8   &-8   &    0 \\
(2, 1)   &(2)      &\ep       &28   &30   &-36 
&-21  &9    &    0 \\
(2, 1)   &(2)      &(1)      &152  &164  &-181
&-99  &-17  &    0 \\
(2, 1)   &(2)      &(2)      &382  &442  &-477
&-256 &-256 &    0 \\
(2, 1)   &(2)      &(1, 1)   &382  &378  &-371
&-182 &-182 &    0 \\
(2, 1)   &(2)      &(1, 1, 1)&610  &472  &-394
&-158 &-394 &    0 \\
(2, 1)   &(1, 1)   &\ep       &28   &26   &-28 
&-15  &11   &    0 \\
(2, 1)   &(1, 1)   &(1)      &152  &140  &-139
&-69  &1    &    0 \\
(2, 1)   &(1, 1)   &(1, 1)   &382  &330  &-293
&-128 &-128 &    0 \\
(2, 1)   &(2, 1)   &\ep       &74   &74   &-81 
&-81  &30   &    0 \\
(2, 1)   &(2, 1)   &(1)      &428  &428  &-433
&-433 &-5   &    0 \\
(2, 1)   &(2, 1)   &(2)      &1168 &1242
&-1218&-1218&-597 &    0 \\
(2, 1)   &(2, 1)   &(1, 1)   &1168 &1094 &-982
&-982 &-435 &    0 \\
(2, 1)   &(2, 1)   &(2, 1)   &3933 &3933
&-3470&-3470&-3470&1     \\
(2, 1)   &(2, 1)   &(1, 1, 1)&2037 &1609
&-1221&-1221&-1221&1     \\
(2, 1)   &(1, 1, 1)&\ep       &38   &26   &-24 
&-24  &15   &    0 \\
(2, 1)   &(1, 1, 1)&(1)      &224  &160  &-136
&-136 &24   &    0 \\
(2, 1)   &(1, 1, 1)&(1, 1)   &610  &444  &-338
&-338 &-116 &    0 \\
(2, 1)   &(1, 1, 1)&(1, 1, 1)&1056 &736  &-480
&-480 &-480 &    0 \\
(1, 1, 1)&\ep       &\ep       &2    &    0&    0&
0&    0&    0 \\
(1, 1, 1)&(1)      &\ep       &8    &2    &-2   &
0&1    &    0 \\
(1, 1, 1)&(1)      &(1)      &38   &17   &-15 
&2    &2    &-1    \\
(1, 1, 1)&(1, 1)   &\ep       &16   &10   &-8  
&-3   &7    &    0 \\
(1, 1, 1)&(1, 1)   &(1)      &84   &54   &-42 
&-15  &12   &    0 \\
(1, 1, 1)&(1, 1)   &(1, 1)   &204  &134  &-97 
&-30  &-30  &    0 \\
(1, 1, 1)&(1, 1, 1)&\ep       &22   &18   &-12 
&-12  &15   &    0 \\
(1, 1, 1)&(1, 1, 1)&(1)      &122  &88   &-59 
&-59  &29   &2     \\
(1, 1, 1)&(1, 1, 1)&(1, 1)   &320  &206  &-130
&-130 &-27  &    0 \\
(1, 1, 1)&(1, 1, 1)&(1, 1, 1)&544  &322  &-175
&-175 &-175 &-4
\end{array}
\]
\caption{Table of the coefficients of the paper, for three indexing partitions with weight at most $3$ (part 2 of 2).}
\label{part2}
\end{table}


\section{Bounds}\label{appbounds}

We prove here the assertions made in Remarks \ref{rem:ki} and \ref{rem:3row} about the values of the constants $k_i$ (in Theorem \ref{3columns}) and $k'_i$ (in Theorem \ref{2row}). These technical and less central results do not appear in the printed version of this work.

\subsection{Hook stability, reduced Kronecker coefficients}

In this section, we  find explicitly bounds for the quantities  $k_1$, $k_2$, $k_3$ appearing  in Theorem \ref{3columns}.

\begin{theorem} \label{bounds-columns}
In Theorem \ref{3columns}, one can take 
\[
\left\lbrace
\begin{array}{rcl}
k_1 &=& |\alpha|+\alpha_1+\beta'_1+\gamma'_1, \\
k_2 &=& |\beta|+\beta_1+\alpha'_1+\gamma'_1,   \\
k_3 &=& |\gamma|+\gamma_1+\alpha'_1+\beta'_1
\end{array}
\right.
\]
\end{theorem}

\begin{proof}
After the proof of Theorem \ref{3columns}, one can take $k_i = \max_{\omega \in \Omega} \ell_i(\omega)$, where $\Omega$ is the support of $\polcol=\polcol_{\alpha,\beta,\gamma}(x,y,z)$.

Let us perform the change of variables $x=\frac{vw}{u}$, $y=\frac{uw}{v}$, $z=\frac{uv}{w}$, so that the identity $x^a y^b z^c = u^{\ell_1} v^{\ell_2} w^{\ell_3}$ holds. Then $k_1$ (resp. $k_2$, $k_3$), as defined above, is the degree of $P$ with respect to the variable $u$ (resp. $v$, $w$). 

After this change of variables, $H(-\varepsilon x,-\varepsilon y,-\varepsilon z)$ equals
\begin{multline*}
XYZ+XY+XZ+YZ
-\varepsilon \cdot\left( \frac{uv}{w} XY +\frac{uw}{v} XZ +\frac{vw}{u} YZ  \right)\\
+ u^2 X + v^2 Y + w^2 Z
+ \varepsilon \frac{1}{vw}\left( u^2-u v^2 - u w^2\right) X\\
+ \varepsilon \frac{1}{uw}\left( v^2-v u^2 - v w^2\right) Y
+ \varepsilon \frac{1}{uv}\left( w^2-w u^2 - w v^2\right) Z.
\end{multline*}
We reorder the terms as follows:
\[
H(-\varepsilon x,-\varepsilon y,-\varepsilon z)=u^2 X 
+ \varepsilon u X H_1 
-\varepsilon u \left(\frac{v}{w} Y+\frac{w}{v} Z \right) + H_0.
\]
where $H_0$ is a sum of monomials with non-positive degree in $u$, and $H_1$ is free of $u$ and $X$. We now factorize $\sigma[H(-\varepsilon x,-\varepsilon y,-\varepsilon z)]$ as
\[
\sigma[u^2 X] \cdot \sigma[\varepsilon u X H_1] \cdot \sigma\left[-\varepsilon u \frac{v}{w} Y \right]
\cdot \sigma\left[-\varepsilon u \frac{w}{v} Z \right] \cdot \sigma[H_0]
\]
 and expand each series, except $\sigma[H_0]$. We get that $\sigma[H(-\varepsilon x,-\varepsilon y,-\varepsilon z)]$ is equal to  
\[
\sum u^{2i} h_i[X]\, u^j e_j[X H_1]\, \left(\frac{v}{w}u\right)^k e_k[Y]\, \left(\frac{w}{v}u\right)^{\ell} e_\ell[Z] \sigma[H_0],
\]
where the sum ranges over all nonnegative integers $i$, $j$, $k$, $\ell$. Therefore,
\begin{multline*}
\polcol=\scalar{\sigma[H(-\varepsilon x,-\varepsilon y,-\varepsilon z)]}{s_{\alpha}[X] s_{\beta}[Y] s_{\gamma}[Z]}\\
= \sum u^{2i+j+k+\ell} v^{k-\ell} w^{\ell-k} \scalar{e_j[X H_1] \sigma[H_0]}{(h_i^{\perp} s_{\alpha})[X] (e_k^{\perp} s_{\beta})[Y] (e_{\ell}^{\perp} s_{\gamma}[Z])}.
\end{multline*}
We have $h_i^{\perp} s_{\alpha}=0$ unless $i \leq \alpha_1$, and that  $e_k^{\perp} s_{\beta}=0$ (resp. $e_\ell^{\perp} s_{\gamma}=0$) unless $j \leq \beta'_1$ (resp. $k \leq \gamma'_1$). Finally, since $e_j[X H_1]$ is homogeneous of degree $j$ in $X$ and $h_i^{\perp} s_{\alpha}$ has degree $|\alpha|-i$, the summand corresponding to $i$, $j$, $k$, $\ell$ can be non zero only if $i \leq \alpha_1$, $j \leq |\alpha|-i$, $k \leq \beta'_1$ and $\ell \leq \gamma'_1$. Therefore $2i+j+k+\ell \leq |\alpha|+\alpha_1+\beta'_1+\gamma'_1$.

This proves that in Theorem \ref{3columns}, one can take $k_1=|\alpha|+\alpha_1+\beta'_1+\gamma'_1$. By symmetry, it follows that one can also take $k_2=|\beta|+\beta_1+\alpha'_1+\gamma'_1$ and $k_3=|\gamma|+\gamma_1+\alpha'_1+\beta'_1$.
\end{proof}

\begin{remark}
More detailed computations show that the coefficient of $u^{|\alpha|+\alpha_1+\beta'_1+\gamma'_1}$ in $\polcol_{\alpha,\beta,\gamma}$ is
\[
s_{{\overline{\alpha}}'}\left[ \frac{1+v^2+w^2}{vw}\right] \cdot s_{\beta'}[1] s_{\gamma'}[1] 
\]
where $\overline{\alpha}$ is the partition obtained from $\alpha$ by removing its first row (and ${\overline{\alpha}}'$ is the conjugate of $\overline{\alpha}$). This is non--zero if and only if $\beta$ and $\gamma$ have at most one column, and $\overline{\alpha}$ has at most three columns. This is, the only case when the bound is reached. 
\end{remark}

\subsection{First row for reduced Kronecker coefficients.}  

 We give bounds for the constants $k'_1, k'_2$ and $k'_3$ appearing In Theorem \ref{2row}.

\begin{theorem} \label{effective bounds 2row}
In Theorem \ref{2row}, one can take 
\[
\left\lbrace
\begin{array}{rcl}
k'_1 &=& |\alpha|+|\beta|+|\gamma| + \beta_1 , \\
k'_2 &=&  |\alpha|+|\beta|+|\gamma| + \gamma_1 , \\
k'_3 &=&  |\alpha|+|\beta|+|\gamma| + \alpha_1+ \beta_1 + \gamma_1.
\end{array}
\right.
\]
\end{theorem}

\begin{proof}
After the proof of Theorem \ref{2row}, one can take 
\[
\left\lbrace
\begin{array}{rcl}
k'_1 &=& \max_{\omega \in \Omega} \ell_1(\omega),\\
k'_2 &=&  \max_{\omega \in \Omega} \left( \ell_1(\omega)-\ell_2(\omega)\right), \\
k'_3 &=& \max_{\omega \in \Omega} \left( \ell_1(\omega)-\ell_3(\omega)\right).
\end{array}
\right.
\]
Let us perform the change of variables $x = u v w$,  $ y = v w$, $z = u w$, so that $x^a y ^b z^c = u^{\ell_1-\ell_2} v^{\ell_1-\ell_3} w^{\ell_1}$. Then the constants $k'_1$, $k'_2$, $k'_3$ are the degrees of $\polrow_{\alpha, \beta, \gamma}$ in the variables, respectively, $u$, $v$ and $w$.

Let us bound the degree in $u$ of  $\polrow$. 
After the change of variables, we obtain that
$
H(x,y,z)
=
u^2 v w^2 Y + u H_1 + H_0
$
where $H_1$ is free of $u$ and has all its terms of degree at least $1$ in $X$, $Y$ and $Z$, and $H_0$ has all its term of degree $\leq 0$ in $u$. 
Thus,
\begin{align*}
\sigma[H]
= \sigma[u^2 v w^2 Y]  \sigma[u H_1]  \sigma[H_0]
= \sum_{i,j} u^{2 i + j} (v w^2)^i h_i[Y] h_j[H_1] \sigma[H_0] 
\end{align*}
and, therefore,
\begin{align*}
\polrow{} &= \sum_{i,j} u^{2i+j} (v w^2)^j \scalar{ h_i[Y] h_j[H_1] \sigma[H_0]}{s_{\alpha}[X] s_{\beta}[Y] s_{\gamma}[Z]}\\
&=\sum_{i,j} u^{2i+j} (v w^2)^j \scalar{h_j[H_1] \sigma[H_0]}{s_{\alpha}[X] (h_i^{\perp }s_{\beta})[Y] s_{\gamma}[Z]}\\
\end{align*}
Note that $h_i^{\perp} s_{\beta} = 0$ unless $i \leq \beta_1$. Moreover the left--hand side of each scalar product in the sum is now a sum of homogeneous symmetric functions all of total degree at least  $j$, while the right--hand side has degree $|\alpha|+ |\beta|+ |\gamma| -i$. Thus,  the non-zero summands  fulfill $j \leq |\alpha|+|\beta+|\gamma| -i$.  We conclude that for all non--zero summands, $2i + j \leq |\alpha|+ |\beta|+ |\gamma|+ \beta_1$. 
\end{proof}


\section{Another approach to the hook stability property, derived from Murnaghan's stability and conjugation}

\subsection{Half of Theorem \ref{thm:hook stab}}

It is well--known that the Kronecker coefficients are invariant under conjugation of any two of their arguments:
\[
\K{\lambda}{\mu}{\nu}=\K{\lambda'}{\mu'}{\nu}=\K{\lambda'}{\mu}{\nu'}=\K{\lambda}{\mu'}{\nu'}.
\]
We have (conjugating the arguments in position 1 and 2), $\K{\lambda}{\mu}{\nu}= \K{\lambda'}{\mu'}{\nu}$. 
Assuming that $(\lambda',\mu',\nu)$ is stable,  that is, that the value of the Kronecker coefficient does not change by adding one to the first row of each, we have
\begin{equation}\label{stab1}
\K{\lambda'}{\mu'}{\nu} = \K{\lambda'\op{1}{0}}{\mu'\op{1}{0}}{\nu\op{1}{0}}.
\end{equation}
Conjugating again the arguments in position 1 and 2, we have that 
$
\K{\lambda'\op{1}{0}}{\mu'\op{1}{0}}{\nu\op{1}{0}}=\K{\lambda\op{0}{1}}{\mu\op{0}{1}}{\nu\op{1}{0}}.
$
We conclude that 
$
\K{\lambda}{\mu}{\nu} = \K{\lambda\op{0}{1}}{\mu\op{0}{1}}{\nu\op{1}{0}}.
$
From Lemma \ref{lemma:N0} we get the following sufficient condition for \eqref{stab1} to hold:
\begin{equation}\label{stabconj1}
N \geq N_0(\cutt{\lambda}, \cutt{\mu}, \cutt{\nu})+  (\lambda'_1+\mu_1+\nu_1)/2,
\end{equation}
where, again, $N$ is the weight of the partitions $\lambda$, $\mu$ and $\nu$.
Likewise, by conjugating the partitions at positions 1 and 3, or 2 and 3, or the Kronecker coefficients, we would get that 
$
\K{\lambda}{\mu}{\nu} = \K{\lambda\op{0}{1}}{\mu\op{1}{0}}{\nu\op{0}{1}}$ and $ \K{\lambda}{\mu}{\nu} = \K{\lambda\op{1}{0}}{\mu\op{0}{1}}{\nu\op{0}{1}},
$
under the assumptions that, respectively
\begin{align}
N &\geq N_0(\cutt{\lambda}, \cutt{\mu}, \cutt{\nu})+  (\lambda_1+\mu'_1+\nu_1)/2, \label{stabconj2}\\
N &\geq N_0(\cutt{\lambda}, \cutt{\mu}, \cutt{\nu})+  (\lambda_1+\mu_1+\nu'_1)/2. \label{stabconj3}
\end{align}
Of course we have also $\K{\lambda}{\mu}{\nu} =\K{\lambda\op{1}{0}}{\mu\op{1}{0}}{\nu\op{1}{0}} $ when
\begin{equation}\label{stabconj0}
N \geq N_0(\cutt{\lambda}, \cutt{\mu}, \cutt{\nu})+  (\lambda_1'+\mu'_1+\nu'_1)/2.
\end{equation}

 Assume that all four hypotheses \eqref{stabconj1},\eqref{stabconj2}, \eqref{stabconj3} and \eqref{stabconj0} hold for $(\lambda, \mu, \nu)$. One can check that they hold as well for all triples of partitions $(\lambda \op{m-a}{a}, \mu\op{m-b}{b}, \nu\op{m-c}{c})$ such that $(a,b,c,m)$ is in the semigroup  $\mathcal{S}$  generated by the four vectors $(1,1,0,1)$, $(1,0,1,1)$, $(1,1,0,1)$ and $(0,0,0,1)$. It follows, by induction, that 
\[
\K{\lambda}{\mu}{\nu}=\K{\lambda \op{m-a}{a}}{\mu\op{m-b}{b}}{\nu\op{m-c}{c}}
\]
for all $(a,b,c,m) \in \mathcal{S}$.  It is not difficult to establish (using again the change of variables $\ell_1$, $\ell_2$, $\ell_3$) that the semigroup $\mathcal{S}$ is the set of points $(a,b,c,m) \in \NN^4$ that fulfill the condition \eqref{coneD} \emph{and, additionally, the congruence $a+b+c \equiv 0 \mod 2$}.

Let us consider also the case when the triple $(\lambda, \mu, \nu)$ is not assumed to fulfill \eqref{stabconj1},\eqref{stabconj2}, \eqref{stabconj3} and \eqref{stabconj0}. Then one can check that $(\lambda \op{m-a}{a}, \mu\op{m-b}{b}, \nu\op{m-c}{c})$  fulfill these conditions when a system of type \eqref{translation of E} holds, with 
\[
\left\lbrace
\begin{matrix}
d_1  &=& 2 \; (N_0-N) + \lambda'_1+\mu_1+\nu_1, \\
d_2  &=& 2 \; (N_0-N) + \lambda_1+\mu'_1+\nu_1, \\
d_3  &=& 2 \; (N_0-N) + \lambda_1+\mu_1+\nu'_1, \\
d      &=& N_0 -N  + (\lambda'_1+\mu'_1+\nu'_1)/2.
\end{matrix}
\right.
\]
Therefore, we \emph{nearly} recover the stability property of Theorem \ref{thm:hook stab}. To recover the full theorem, it would be enough to establish that there exists $m$ big enough, such that
\[
\K{\lambda\op{2m}{2m}}{\mu\op{2m}{2m}}{\nu\op{2m}{2m}} 
=
\K{\lambda\op{2m+1}{2m+1}}{\mu\op{2m+1}{2m+1}}{\nu\op{2m+1}{2m+1}} .
\]
See next subsection for a possible approach to this question.

\subsection{A monotonicity conjecture}

Remember that Murnaghan's sequences of Kronecker coefficients are weakly increasing (see Section \ref{Murnaghan}): for any three partitions  $\lambda$, $\mu$, $\nu$ of the same weight,
\begin{equation}\label{ineq 111}
\K{\lambda}{\mu}{\nu} \leq \K{\lambda+(1)}{\mu+(1)}{\nu+(1)}.
\end{equation}
Here is, conjecturally, a more general monotonicity property.
\begin{conjecture}[Conjecture \ref{conj mono} restated]
For any three partitions $\lambda$, $\mu$ and $\nu$ of the same weight, and any $(a,b,c,m)$ fulfilling \eqref{coneD},
\[
\K{\lambda}{\mu}{\nu} \leq \K{\lambda\op{m-a}{a}}{\mu \op{m-b}{b}}{\nu\op{m-c}{c}}.
\]
\end{conjecture}
Again, using the symmetries of the Kronecker coefficients, it is not difficult to prove part of this conjecture.
\begin{proposition}\label{prop conj}
Let $\lambda$, $\mu$ and $\nu$ be three non--empty partitions of the same weight. 
If 
$
\K{\lambda}{\mu}{\nu} \leq \K{\lambda\op{1}{1}}{\mu\op{1}{1}}{\nu\op{1}{1}},
$
then
\[
\K{\lambda}{\mu}{\nu} \leq \K{\lambda\op{m-a}{a}}{\mu\op{m-b}{b}}{\nu\op{m-c}{c}}
\]
for all $(a,b,c,m)$ fulfilling \eqref{coneD}.
\end{proposition}

\begin{proof}
We use again the symmetry of the Kronecker coefficients under conjugating two arguments. 
We have the identity $\K{\lambda}{\mu}{\nu}= \K{\lambda'}{\mu'}{\nu}$. Using \eqref{ineq 111} we get
$
\K{\lambda'}{\mu'}{\nu} \leq \K{\lambda'\op{1}{0}}{\mu'\op{1}{0}}{\nu\op{1}{0}}.
$
Conjugating again the arguments in position 1 and 2, we have
that $
\K{\lambda'\op{1}{0}}{\mu'\op{1}{0}}{\nu\op{1}{0}}$ is equal to $\K{\lambda\op{0}{1}}{\mu\op{0}{1}}{\nu\op{1}{0}}.
$
Therefore, $
\K{\lambda}{\mu}{\nu} \leq \K{\lambda\op{0}{1}}{\mu\op{0}{1}}{\nu\op{1}{0}}.
$
Likewise
$
\K{\lambda}{\mu}{\nu} \leq \K{\lambda\op{0}{1}}{\mu\op{1}{0}}{\nu\op{0}{1}} $   and $\K{\lambda}{\mu}{\nu} \leq \K{\lambda\op{1}{0}}{\mu\op{0}{1}}{\nu\op{0}{1}}.$

Using these 3 identities, together with \eqref{ineq 111}, we see that $\K{\lambda}{\mu}{\nu} \leq \K{\lambda}{\mu}{\nu}(\tau)$ for all $\tau$ in the semigroup $\mathcal{S}$ generated by $(1,1,0,1)$, $(1,0,1,1)$, $(0,1,1,1)$ and $(0,0,0,1)$ (here we use again the notation introduced after \eqref{stabconj0}). This semigroup was determined earlier: the points $(a,b,c,N)$ that fulfill \eqref{coneD} are exactly the elements of $\mathcal{S}$ and the elements $(1,1,1,1)+\tau$  for $\tau \in \mathcal{S}$. This proves the proposition.
\end{proof}

Proposition \ref{prop conj} shows that Conjecture \ref{conj mono} is equivalent to the following 
seemingly much weaker statement. 
\begin{conjecture}[Equivalent form of Conjecture \ref{conj mono}; this is Conjecture \ref{conj 111} restated]
For any three partitions $\lambda$, $\mu$ and $\nu$ of the same weight,
\[
\K{\lambda}{\mu}{\nu} \leq \K{\lambda\op{1}{1}}{\mu\op{1}{1}}{\nu\op{1}{1}}.
\]
\end{conjecture}

\begin{remark}
Conjecture \ref{conj 111} was checked by computer, with SAGE \cite{SAGE}, for all triples of partitions of weight at most $16$.
\end{remark}

\begin{remark}\label{alternate proof}
A proof of Conjecture \ref{conj 111} would provide an alternative proof of Theorem \ref{thm:hook stab}.
\end{remark}

\section{One the generating function for the coefficients $B_{\alpha,\beta,\gamma}$}

\subsection{Expression involving Schur functions indexed by hooks}

It is proved in Theorem \ref{prop:ABC} that the Schur generating function for the coefficients  $B_{\alpha,\beta,\gamma}$ is
\[
\sigma[XYZ+2W] \cdot \left( \frac{3}{4}+\frac{1}{4}\sigma[(\varepsilon-1) W]-\frac{1}{2} \chi[W] + \chi[YZ-X] \right)
\]
The following result is stated in a remark, with no proof,

\begin{proposition}
Fix partitions $\alpha$, $\beta$, $\gamma$. 
The coefficient $B_{\alpha,\beta,\gamma}$ in Theorem \ref{2row} is the coefficient of $s_{\alpha}[X] s_{\beta}[Y] s_{\gamma}[Z]$ in the expansion in the Schur basis of
\[
\sigma[XYZ+2W]  \cdot \left(1 - \sum_{a \text{even }, b} (-1)^b s_{(a|b)}[W] + \sum_{a, b} (-1)^b s_{(a|b)}[YZ-X] \right).
\]
\end{proposition}

\begin{proof}
From Cauchy Formula,
\begin{multline}
\sigma[(\varepsilon-1) W] = 
\sigma[(1-\varepsilon) (-W)] = \\
\sum_{\lambda} s_{\lambda}[1-\varepsilon] s_{\lambda}[-W]=
\sum_{\lambda} s_{\lambda}[1-\varepsilon] (-1)^{|\lambda|}s_{\lambda'}[W].
\end{multline}
After \cite[Ex. 7.43 with $t=1$]{Stanley}, $s_{\lambda}[1-\varepsilon]$ is $1$ if $\lambda$ is the empty partition, $2$ if $\lambda$ is a hook and $0$ otherwise. Therefore,
\[
\sigma[(\varepsilon-1) W] = 1+2\;\sum_{a, b \geq 0} (-1)^{1+a+b} s_{(a|b)}[W].
\]
Thus,
\[
\frac{3}{4}+\frac{1}{4}\sigma[(\varepsilon-1) W] = 1 + \frac{1}{2} \sum_{a, b} (-1)^{1+a+b} s_{(a|b)}[W].
\]
After \cite[I.\S 3. Ex. 11 (2) with $\mu=\ep$]{Macdonald}, we have 
\begin{equation}\label{chi}
\chi=\sum_{a, b} (-1)^{b} s_{(a|b)}.
\end{equation}
Therefore,
\begin{multline*}
\frac{3}{4}+\frac{1}{4}\sigma[(\varepsilon-1) W] -\frac{1}{2} \chi[W] 
=\\
1+ \frac{1}{2} \sum_{a, b} (-1)^{1+a+b} s_{(a|b)}[W] -\frac{1}{2}\sum_{a, b} (-1)^{b} s_{(a|b)}\\
=
1 -  \sum_{a \text{ even}, b} (-1)^b s_{(a|b)}[W]
\end{multline*}
Using again \eqref{chi} to rewrite $\chi[YZ-X]$, we get the following Formula for the generating function of the coefficients of $B$:
\[
\sigma[XYZ+2W]  \cdot \left(1 - \sum_{a \text{even }, b} (-1)^b s_{(a|b)}[W] + \sum_{a, b} (-1)^b s_{(a|b)}[YZ-X] \right)
 \]
\end{proof}

\subsection{Toolbox for other expressions}

In order to write in other ways the generating function for the coefficients $B_{\alpha,\beta, \gamma}$, the following formulas may be useful:
\begin{align*}
\sigma[X] \cdot \chi[X] &= \sum_{k} k \, h_k[X],\\  
\sigma[2X] \cdot \chi[X] &= \sum_{\lambda: \ell(\lambda) \leq 2} \frac{(\lambda_1-\lambda_2+1)(\lambda_1+\lambda_2)}{2} \, s_{\lambda}[X].
\end{align*}
They follow from the fact that $\sigma[t X] \chi[tX]$ is the derivative of $\sigma[tX]$ (for the first one), and that $\sigma[2 t X] \chi[ 2 tX]$ is the derivative of $\sigma[2 tX]$. Last, by Cauchy formula,
\[
\sigma[2 t X]= \sum_{\lambda} s_{\lambda}[2] s_{\lambda}[X] t^{|\lambda|}, 
\]
and $s_{\lambda}[2] = (\lambda_1-\lambda_2+1)$ is $\lambda$ has at most two parts, and is equal to $0$ otherwise.

\end{document}